\documentclass[11pt,leqno]{amsart}
\usepackage{graphicx}
\usepackage{indentfirst,csquotes}

\usepackage[margin=1in]{geometry}

\usepackage{amssymb,amsthm,amsmath}
\usepackage{xcolor,paralist,hyperref,fancyhdr,etoolbox}
\newtheorem{theorem}{Theorem}
\newtheorem{definition}{Definition}
\newtheorem{example}{Example}
\newtheorem{assumption}{Assumption}
\newtheorem{lemma}{Lemma}
\newtheorem{proposition}{Proposition}
\newtheorem{corollary}{Corollary}

\usepackage{enumitem}
\usepackage{soul}
\usepackage{color-edits}
\usepackage{amstext,amsfonts,amssymb}%amsthm
\usepackage{graphicx,color}
\usepackage[normalem]{ulem}
\usepackage{algorithm}
\usepackage{algpseudocode}
\usepackage{tikz}
\usepackage{comment}
\usepackage{natbib}
 \bibpunct[, ]{(}{)}{,}{a}{}{,}%
\hypersetup{ colorlinks=true, linkcolor=black, filecolor=black, urlcolor=black }

\usepackage{lipsum}

\begin{document}

\title[Drift Optimization of Stochastic Models]{Drift Optimization of Regulated Stochastic Models\\Using Sample Average Approximation}
\author[Z.Zhou]{Zihe Zhou}
\author[H.Honnappa]{Harsha Honnappa}
\author[R.Pasupathy]{Raghu Pasupathy}

\date{\today}
\address{315 N Grant St, West Lafayette, IN 47907-2023, United States}
\email{zhou408@purdue.edu}
\maketitle

\let\thefootnote\relax

\begin{abstract}
This paper introduces a {\it drift optimization} model of stochastic optimization problems driven by regulated stochastic processes. A broad range of problems across operations research, machine learning and statistics can be viewed as optimizing the ``drift'' associated with a process by minimizing a cost functional, while respecting path constraints imposed by a Lipschitz continuous regulator. Towards an implementable solution to such infinite-dimensional problems, we develop the fundamentals of a Sample Average Approximation (SAA) method that incorporates (i) path discretization, (ii) function-space discretization, and (iii) Monte Carlo sampling, and that is solved using an optimization recursion such as mirror descent. We start by constructing pathwise directional derivatives for use within the SAA method, followed by consistency and complexity calculations. The characterized complexity is expressed as a function of the number of optimization steps, and the computational effort involved in (i)--(iii), leading to guidance on how to trade-off the computational effort allocated to optimization steps versus the ``dimension reduction'' steps in (i)--(iii).
\end{abstract} %%%%%%%%%

\bigskip

\section{Introduction}\label{sec:Intro}
This paper introduces and studies a {\it drift optimization} model of a broad class of stochastic optimization problems driven by regulated stochastic processes. 
A range of problems, spanning operations research, machine learning and statistics, can be modeled as

\begin{align*}\label{opt}
\tag{OPT}
 &~\text{min:}~J(F) := \underset{C \times \mathbb R^d}\int \tilde J\circ\Gamma(z+F)\pi^x(dz) \otimes \mu_0(dx),\\
\nonumber&\text{subject to:}~
%&\qquad\qquad  X^F_t(z) = x + F(t) + \int_0^t \sigma(s) dB_s(z) + L(t) ~\forall~t \geq 0\\
F \in \mathcal F,
\end{align*}
where $C$ is the space of $\mathbb{R}^d$-valued continuous functions having compact domain $[0,T]$, and equipped with the norm $\|x\| = \max_{1 \leq i \leq d} \sup_{t \in [0,T]} |x_i(t)|$, the feasible set $\mathcal F$ is a convex, compact subset of $C$, and $\pi^x$ is the path measure corresponding to the $\mathbb{R}^d$-valued stochastic process $Z = (Z_t : t \in [0,T])$ with $Z_0$ having distribution $\mu_0$.  The ``regulator'' $\Gamma : C \to C$ is a Lipschitz continuous function that maps the process $Z + F$ to a subset of $\mathbb{R}^d$ (such as an orthant), and $\tilde J : C \to \mathbb R$ is a real-valued cost functional. 

%We develop an implementable Sample Average Approximation (SAA) method for solving drift optimization problems of the type~\eqref{opt} using an optimization recursion, e.g., mirror descent, alongside finite-dimensional approximations to handle path discretization, function space discretization, and Monte Carlo sampling. , derive pathwise directional derivatives, and characterize consistency and convergence rates through . Our methodology offers a novel and generalizable approach for solving drift optimization problems with regulated dynamics.
%\subsubsection{Stochastic Optimization}

Observe that (OPT) could be solved as a (stochastic) optimal control problem, typically solved using dynamic programming. This requires the solution of the corresponding Hamilton-Jacobi-Bellman (HJB) equation, and yields ``closed loop'' optimal control policies; for more discussion on this see Section~\ref{sec:prior}.

In contrast, the feasible set in (OPT) is {\it not} closed loop. Adopting a mathematical programming approach, we view (OPT) as a (single-stage) stochastic optimization problem over a path-valued (or infinite dimensional) feasible space. From that perspective, stochastic optimization over Banach space-valued variables is a closely related formalism to the current setting. SAA for finite-dimensional single and multistage stochastic optimization has been extensively studied; therefore, we primarily refer to standard textbooks on this topic~\citep[Chapter 8]{Shapiro_2021}. On the other hand, there is a surprising dearth of studies on SAA over more general spaces, and most of the extant work focuses on PDE constrained optimization. We highlight, in particular, the recent work of~\cite{Milz2023,Milz2023b}, which considers a regularized stochastic optimization problem over an infinite-dimensional reflexive Banach space. These papers establish consistency of the Clarke-stationary points of the corresponding SAA, and the method is applied to a semi-linear PDE constrained optimization problem. On the one hand, drift optimization is more specific because of the focus on additive shifts of the driving stochastic process. On the other hand, the model also allows for general constraint sets, though any eventual duality results will require the definition of constraints through operator (in)equalities. %We also observe that the SAA analysis in this paper is considerably more general and detailed.% (see points 2 and 3 below). 

\subsection{Examples}
~\eqref{opt} subsumes a wide range of optimization problems, and is best illustrated through the following examples that arise in queueing, probabilistic machine learning, and optimal transport, respectively.

\begin{example}~\label{ex:rbm-opt} Let $\mathcal F = W_0^{1,2}$ be the Sobolev space consisting of $\mathbb R$-valued absolutely continuous functions with $L^2$-integrable derivatives and initial value 0. Let $Z = \sigma B$, where $B$ is a Wiener process with measure $\pi^x$ and $\sigma > 0$, and let $\mu_0 = \delta_0$ be the Dirac measure concentrated at zero.
%then $({C},\mathcal{F},\pi^x)$ is the classic Cameron-Martin-Wiener space.
When $\Gamma$ is chosen to be the Skorokhod regulator map that we later elaborate, the random variable $X^F := \Gamma(Z+F)$ is a so-called {\it reflected Brownian motion} (RBM) with drift $F$. Consider a cost functional over $z \in C$, $z \mapsto \tilde{J}(z):= a_1 \int_0^T g(z(s)) ds + a_2 G(z(T))$, where $(a_1,a_2) \in \mathbb R^2$, $g:\mathbb R \to \mathbb R$ and $G : \mathbb R \to \mathbb R$ are well-defined functions. The corresponding optimization problem with decision variable $F$ defines a class of {``open-loop''} optimal control problems over the Banach space $W_0^{1,2}$, driven by an RBM. This class of problems arises in nonstationary queueing network control, scheduling, and inventory control. See~\cite{armony-atar-honnappa} for a diffusion optimization problem over non-decreasing functions $F$, studied in the context of appointment scheduling problems in queueing systems. 
\end{example}

\begin{example} Consider the \emph{information projection} problem:
    $\inf_{\pi \in \mathcal P} \big\{ \mathbb{E}_{\pi}[\Phi] + D_{KL}(\pi\|\pi_0)\big\},$
    where $\Phi$ is a real-valued functional on $C$, $\pi_0$ is the Wiener measure, $\mathcal{P}$ is the class of Gaussian measures that are equivalent to $\pi_0$, and $D_{KL}(\pi\|\pi_0) = \int \log \frac{d \pi}{d \pi_0} \, d\pi$ is the Kullback-Leibler divergence between $\pi$ and $\pi_0$. The measure $\pi^*$, that is equivalent to $\pi_0$, has density \[\frac{d \pi^*}{d \pi_0} = \frac{e^{-\Phi}}{\mathbb E_{\pi_0}[e^{-\Phi}]}.\] The solution of the information projection problem is the `closest' Gaussian measure to the (possibly non-Gaussian) $\pi^*$.
    Information projection is therefore intimately connected with the Schr\"odinger Bridge problem~\citep{leonard2014survey} and optimal transport~\citep{villani2009optimal}, as well as variational inference methods in computational statistics. \citet[Theorem 1.1]{selk2021information} show that this information projection problem  is equivalent to an `open-loop' (or state independent) KL-weighted control~\citep{bierkens2014explicit} of the form
    \[\inf_{F \in \mathcal{F}} \{\mathbb E_{\pi_0}[\Phi(z + F)] + D_{KL}(\pi_F \|\pi_0)\},\] where $\pi_F$ corresponds to the pushforward of $\pi_0$ under the `shift' $F$. It is immediate that this problem is of the form~(OPT), with $\tilde J(z+F) = \Phi(z+F) + D_{KL}(\pi_F\|\pi_0)$, $\Gamma(x) = x$, and $\mu_0 = \delta_{0}$.
    %\rpcomment{Isn't $D_{KL}(\pi_F\|\pi_0)$ zero, under a simple translation?}
    %\hl{HH: say more about the convexity and non-compactness of CM space.}
\end{example}

\begin{example}
\cite{mikami2004monge} characterizes the optimal transport map $s^*$ in Monge's problem~\citep{villani2009optimal} through a small noise limit of Doob's $h$-path process (see~\citep[Sec. 39]{rogerswilliams}) with an initial distribution $\mu_0$ and a terminal one $\mu_1$. Specifically, the $h$-path process is the minimizer of 
\begin{align} \label{eq:mikami}
\inf_b \, \mathbb E \left[ \int_0^1 |b(t)|^2 dt \right]
\end{align}
where $b$ is in the class of progressively measureable processes for which the distribution of $X_\epsilon (1) := X_0 + \int_0^1 b(s) ds + \sqrt{\epsilon} B(1)$ is $\mu_1$ and $X_0 \sim \mu_0$. \citep[Theorem 2.2]{mikami2004monge} proves that $X_\epsilon(1) \stackrel{\text{L}^2}{\to} s^*$ as $\epsilon \to 0$, and the optimal drift $b^*_\epsilon$ (roughly speaking) approximates $s^*(X_0) - X_0$ in the L$^2$ sense as $\epsilon \to 0$. ~\citep[Theorem 1.1]{selk2021information} shows that the optimization problem in~\eqref{eq:mikami} can be upper bounded by an information projection of the small noise Brownian motion measure onto the class of Gaussian shift measures, which is equivalent to a drift optimization problem as described in Example~2.
\end{example}

\subsection{Contributions}
Our proposed solution estimator for the infinite-dimensional problem~\eqref{opt} involves constructing and solving the following finite-dimensional SAA counterpart. 

\begin{align}\label{sampathpbpre}
~& \mbox{minimize:} \quad J_{N,h}(F) := \frac{1}{N} \sum_{j=1}^N \tilde{J}\circ \Gamma(Z_{h,j}+F), \nonumber\\ 
~& \mbox{subject to:} \quad F \in \mathcal{F}_n, \tag{MC-$n$-OPT}\\
~&\text{where}\quad Z_{h,j}\overset{\scriptsize \mbox{iid}}{\sim} \pi^x_{h}. \nonumber 
\end{align} Observe that the SAA counterpart in~\eqref{sampathpbpre} accrues errors due to three levels of finite-dimensionalization: (i) a ``path-discretized'' measure $\pi^x_h$ that approximates $\pi^x$; (ii) Monte Carlo samples $Z_{h,j}, j =1,2, \ldots, N$ from $\pi^x_h$ used to approximate the objective; and (iii) a finite-dimensional feasible region $\mathcal{F}_n$ that approximates the true feasible region $\mathcal{F}$. Furthermore,~\eqref{sampathpbpre} cannot be solved to infinite precision, and entails the use of an appropriate optimization routine, resulting in the fourth source of error due to optimization. 

As we detail below, our contributions all directly or indirectly relate to the construction of a solution to~\eqref{sampathpbpre}, or to the assessment of the solution to~\eqref{sampathpbpre} by analyzing the errors in (i)--(iii) along with the optimization error. 

\begin{enumerate}[leftmargin=*]
	\item (Pathwise Derivative) Our first result in Lemma~\ref{lem:chain-rule} derives the directional derivative of the pathwise cost functional $\tilde J \circ \Gamma$ in (OPT). In the setting where $\Gamma$ is the Skorokhod regulator map, our principal result is Proposition~\ref{prop:Sko-pathwise}, in which we first establish the pathwise directional derivative of the Skorokhod regulator at each time point $t \in [0,T]$. This result has been established in the applied probability literature, first in the scalar setting by~\cite{MM} who developed it in the context of a time-varying Markovian queueing model (see also~\cite{Whitt_2002,Honnappa2015}in the setting of non-Markovian queues), and in the multivariate setting by~\cite{MR}. The directional derivative of the Skorokhod regulator was also studied in~\cite{Lipshutz_2018,Lipshutz_2019} in convex polyhedra. 
    
    Our contribution is to derive the directional derivative as a consequence of Danskin's Theorem~\cite[Theorem 7.21]{sha04} from convex analysis. Danskin's Theorem is typically used to solve minimax problems in optimization, and this connection between minimax optimization and the directional derivative of the Skorokhod regulator map in stochastic processes is novel. Besides simplifying and clarifying the derivation of the directional derivative, we believe that our lemma also points to the deeper connection between regulation of stochastic processes and minimax optimization. 

    Building on Lemma 1, we then show that the pathwise directional derivative is an unbiased estimator of the directional derivative of the expected cost functional, in Theorem~\ref{thm:objective-derivative}. This result parallels the sensitivity analysis of regulated diffusions in~\cite{Lipshutz1}, who assume that the drift rate function is parameterized. Importantly, we establish this derivative interchange result without discretizing the domain or reverting to standard arguments for derivative interchange for discrete-time stochastic processes (see~\cite{glasserman1990gradient}, for instance). 
    %By pushing out any discretization to an implementation detail, we ensure that errors do not manifest early on in the analysis of the SAA method. 

\item (Consistency) Along the way to establishing the consistency of the solution to~\eqref{sampathpbpre}, we consider the intermediate \emph{functional SAA} problem:
\begin{align}~\label{saaformulation}~\tag{MC-OPT}
    \text{minimize}&~\frac{1}{N} \sum_{i=1}^N \tilde J \circ \Gamma (Z_i + F),\\
    \text{subject to:}&~F \in \mathcal F, \nonumber
\end{align}
where $N$ is the sample size and $\{Z_i : i = 1,\ldots, N\}$ are i.i.d. samples from $\pi^x$. Our approach first establishes a uniform equiconvergence result over function spaces in Proposition~\ref{thm:equiconv} by showing that the Gaussian complexity~\citep{bartlett2002rademacher} of the SAA estimator of the objective is  proportional to $N^{-1/2}$ (for every $N$), assuming the diameter of the constraint set $\mathcal F$ is bounded. As a direct consequence, we obtain the consistency result in Theorem~\ref{thm:equiconv2}. Again, these results are obtained without appealing to any premature discretization of the path space and therefore represent a generalization of the standard results in the finite-dimensional SAA literature. 

\item The samples $\{Z_j, j \geq 1\}$ (MC-OPT) generally cannot be sampled directly from $\pi^x$. For instance, if $Z_1$ is a Brownian motion, its sample paths can be approximated using Euler-Maruyama or Euler-Milstein schemes~\citep{Asmussen_2007}. In other words, the problem in~\eqref{saaformulation} is ``fictitious'' from the standpoint of computation and a further approximation to~\eqref{saaformulation} is necessary for implementation. Furthermore, since $\mathcal{F}$ might be infinite-dimensional,~\eqref{saaformulation} must be optimized (only) over a finite-dimensional subspace of the constraint set $\mathcal{F}$ to allow computation using a method such as gradient descent. Accounting for such sources of error naturally results in the {\it finite SAA} problem seen in~\eqref{sampathpbpre}. 

Our third contribution pertains directly to~\eqref{sampathpbpre}, and appears as a  convergence rate result (Theorem~\ref{thm:saarate}) that accounts for all sources of error. In particular, Theorem~\ref{thm:saarate} quantifies the expected decay rate of the true optimality gap of a solution obtained by executing \emph{mirror descent} on a finite-dimensional approximation of~\eqref{saaformulation} generated using approximations to $\{Z_j, j \geq 1\}$. The resulting rate clarifies the relationship among four sources of error: (i) numerical optimization error due to the use of an iterative scheme such as mirror descent; (ii) Monte Carlo sampling error;  (iii) path approximation error due to ``time'' discretization; and (iv) projection or approximation error due to the use of a finite-dimensional subspace in lieu of $\mathcal{F}$. 

Theorem~\ref{thm:saarate} naturally leads to a computational effort allocation problem posed in the service of deciding how much relative effort is to be expended toward nullifying the various sources of error, in Section~\ref{sec:allocate}. Our findings reveal that the optimal effort allocation depends on the smoothness of the function space: rougher spaces demand greater investment in function approximation, while smoother spaces permit a larger focus on sampling and optimization.
\end{enumerate}

%While drift optimization is just general enough to encompass a broad class of problems, there is sufficient structure to facilitate the development of efficient simulation-based methods to solve (OPT). Consequently, in the spirit of~\cite{stuart2010inverse}, who observed, in the context of infinite-dimensional Bayesian inverse problems, that a guiding principle is to 
% \begin{center}{\it ...avoid discretization until the last possible moment...[which] is enormously empowering throughout numerical analysis,}\end{center} we first consider an infinite-dimensional functional SAA (MC-OPT in point 2 above), and subsequently introduce discretization to develop a practicable methodology (MC-$n$-OPT in point 3). 

%In this regard, we make three main methodological contributions.

\subsection{Roadmap} The rest of the paper is organized as follows. In Section~\ref{sec:prior} we review existing (stochastic) optimal control formalisms that could be used to solve (OPT), when/if viewed as a control problem. In Section~\ref{sec:def} we consolidate all of the key definitions, assumptions and notations that will be used throughout the paper. 
%\zzcomment{are we talking about Section~\ref{sec:necessity}?} \hl{HH: yes, correct.}
%We formally present the drift optimization problem in Section~\ref{sec:setting}.\zzcomment{We really only did the later in Section~\ref{sec:setting}, problem presentation is in the contribution} 
Following this, we derive the pathwise directional derivative and the interchange result, in Section~\ref{sec:dir-der}. Section~\ref{sec:equi} presents our equiconvergence and consistency result for the functional SAA~\eqref{saaformulation}. Our rate of convergence result for the finite SAA~\eqref{sampathpb} is presented in Section~\ref{sec:rate}. Leveraging this rate of convergence result, we then derive an (asymptotically) optimal fixed budget allocation for solving~\eqref{sampathpb} in Section~\ref{sec:allocate}. Finally, we conclude in Section~\ref{sec:conc} with remarks on future directions. %The Appendix contains all the relevant proofs, as well as an exposition on prior art and positioning of drift optimization within some existing optimal control/functional stochastic optimization formalisms.

\section{Positioning and Prior Art}~\label{sec:prior}
As observed in Section~\ref{sec:Intro}, we view the drift optimization problem (OPT) as a mathematical program to be solved over an infinite dimensional feasible space, and have related (OPT) and the related SAA (MC-OPT) and (MC-n-OPT) with existing SAA models over function spaces. To further position our work in the literature, we will recall a number of optimal control formalisms (each developed in separate research community's) and compare them with the drift optimization model. 

First, we contend that the following three characteristics of (OPT) distinguish it from existing optimal control formalisms:
\begin{enumerate}
    \item[(i)] the optimization variable $F$ is a function that additively changes or ``shifts'' the sample paths of the stochastic process $Z$, and need not depend on the process $Z$ in any way;
    \item[(ii)] the feasible set of functions is merely assumed is a convex subset of the class of continuous functions $C$, and therefore places no smoothness or other regularity conditions; and
    \item[(iii)] the formalism includes regulated stochastic processes, which is typically not of consideration in many existing formalisms.
\end{enumerate}

%\subsubsection{Optimal Control Formalisms}~\label{sec:control-form}
\subsection{Mean field control.} The {\it mean field control} formalism, defined in a series of papers~\citep{E1,E2,E3}, focuses exclusively on ordinary differential equation (ODE) forward system dynamics with randomized initial and terminal conditions. The setting in these papers assumes the existence of a coupled pair of random variables $(X_0,Y_0)$ with distribution measure $\mu_0$, and the objective is to compute the optimal control function $\Theta := (\theta^*_t: t \in [0,T])$ by solving the minimization problem
\begin{align*}
    \inf_{\Theta \in L^\infty([0,T],\mathbb R^d)} &~\mathbb E_{\mu_0} \left[ G(X^\theta_T,Y_0) + \int_0^T g(X^\theta_t,\theta_t) dt \right],\\
    \text{subject to:} &~dX_t^\theta = f(t,X_t^\theta,\theta_t) dt,~X^\theta_0 \stackrel{D}{=} X_0.
\end{align*}
This problem was suggested as a plausible control-theoretic formalism of deep learning training. We will not go into the details here, but direct the interested reader to~\cite{E1} where the ideas were first laid out. In nearly contemporaneous work by~\cite{chen2018neural}, the controlled and randomized forward dynamics above were developed as an asymptotic approximation to ResNets and referred to as neural ordinary differential equations (NODE). Here, we observe that the randomization of the controlled differential equation is entirely due to the random initial condition. The minimization problem above is typically solved using an SAA reformulation. Furthermore, one can fully recover the mean field control setting by viewing the solution space of NODE (parameterized by $\Theta$) as a feasible space of drift functions $\mathcal F$. This is achieved by setting the path measure as $\pi^x(dz) = \delta_0(dz)$ (i.e., the Dirac measure concentrated at the zero path) and removing the Skorokhod regulator $\Gamma$, allowing (OPT) to encompass the mean field control setting described above under an appropriate choice of the cost functional $\tilde J$. 

\subsection{Uncertain optimal control.} Closely related to the mean field control formulation above is the class of {\it uncertain optimal control problems} (UOCPs), introduced by ~\cite{PHELPS20142987,Phelps_2016} and subsequently solved using SAA. The UOCP is a stochastic optimization problem of the form
\begin{align*}
	\text{minimize} &~\mathbb E[G(X^\eta_T(\omega),\omega)],\\
	\text{subject to:} &~\dot{X}^\eta_t = b(X^\eta_t(\omega),u_t,\omega),\\&~X_0^\eta = x,\text{and}~\eta=(x,u),
\end{align*}
where $\eta$ is the optimization variable and $\omega$ is a random variable. In a UOCP, the state process $X^\eta$ is determined as the solution of the randomized ODE in the constraint. Note that the mean field control problem differs from the UOCP by assuming a fixed probability distribution measure over the initial conditions of the forward dynamics, while the initial conditions are optimized in the case of the UOCP. 

Observe that both the mean field and UOCP formalisms yield optimal control functions that are state-independent, akin to the drift optimization formalism. On the other hand, they only consider smooth forward dynamics, and do not consider stochastic process models of the dynamics.  {\it Ensemble control}, next, represents a natural first foray into problems with stochastic dynamics.

% (OPT) has fundamentally different constraints from the control formulations discussed. Nonetheless, (OPT) can be transformed to a form paralleling uncertain optimal control problems (UOCPs)~\cite{Phelps_2016} when $Z$ is an Ito semimartingale. 

\subsection{Ensemble control.} Ensemble control~\citep{brockett,brockettb,Bartsch,MR3880221,Fleig} is a formalism that considers optimal control of an indexed collection of density  or ensemble functions $(f(t,\cdot): t \geq 0)$ that evolve according to a Fokker-Planck equation, by solving
\begin{align*}
    ~&\text{minimize}~J(f,u)\\
    ~&\text{subject to:}\\
    ~&\frac{\partial}{\partial t}f(t,x) + \frac{\partial}{\partial x}b(t,x,u(t)) -\frac{1}{2} \frac{\partial^2}{\partial x^2} f(t,x) = 0.
\end{align*}
Here, $u(t)$ is the optimization variable, and the constraints are the Fokker-Planck equation governing the evolution of the density function; note that we have written this for a scalar problem, but the formalism can be generalized. A clean interpretation of this formalism is obtained when one considers that the control is imposed on an entire population/ensemble rather than a single sample path. Under reasonable regularity conditions, there exists a (controlled) Markov diffusion process whose marginals correspond to the Fokker-Planck equation in the constraint (see~\citep{Mikami1999} and~\cite[Sec. 3.2.1]{MikamiBook}). In this sense, the ensemble control formalism parallels drift optimization wherein the drift influences every sample path in the same way (and hence the entire ensemble). However, we observe that standard ensemble control does not consider boundary conditions on the Fokker-Planck. 

\subsection{Stochastic optimal rate control.}~\label{sec:drc} Suppose the class of drift functions $\mathcal F$ consists of absolutely continuous functions. Then, there exists a locally integrable function $(f(t): t \geq 0)$ such that $F(t) = \int_0^t f(s) ds$. For example, if $f(t) = -\theta t$ for all $t > 0$ with $\theta \in \mathbb R_+$, and $\sigma(t) \equiv \sigma \in \mathbb R_+$ then $X^F(t)$ is a $(\theta,\sigma)$ reflected Brownian motion with drift, as defined in~\cite{Harrison_2013}. In this setting, for small $h > 0$ the instantaneous state of the diffusion, $X^F(t+h) - X^F(t),$  is controlled/influenced by the locally integrable drift rate function $f(t) h$, and drift optimization (OPT) can be equivalently viewed as a stochastic optimal control problem over the class of `open loop' control functions $f$. More generally, when $f$ is an adapted process, this problem has been studied under the rubric of {\it drift rate control} as defined in~\cite{Ata_2005,ata2024drift}, and solved using classical stochastic optimal control techniques. Formally, in the scalar setting, the finite horizon drift rate control problem is given by
\begin{align*}
    \text{minimize}~&~\mathbb E\left [ \alpha_1 \int_0^T g(f(X_t)) dt + \alpha_2 L(T) \right],\\
    \text{subject to:}~&~ dX_t = f(X_t) dt + dZ_t + L(t),~X(0) = x,
\end{align*}
where $L(t) := \sup_{0 \leq s \leq t} (-(x + \int_0^s f(X_r) dr + \int_0^s dZ_r))_+$, $(a)_+ := \max\{a,0\}$, and $x \in \mathbb R$ is fixed. In particular, recent work in~\cite{ata2024drift,ata2024singular} has focused on numerically solving the corresponding Hamilton-Jacobi-Bellman (HJB) equations. However, there is no existing work on using mathematical programming methods to solve for the minimal drift rate control. 

\paragraph{Remark 1.} The sample paths of the forward dynamics in both the mean field and UOCP formalisms are necessarily smooth functions, and therefore differ from (OPT), where the forward dynamics have the regularity of the driving stochastic process $Z$. In Section~\ref{sec:doss-sussmann}, we show that (OPT) can be expressed in a form similar to a UOCP when $Z$ is an Ito process, albeit constrained by a differential equation whose solutions are `rough'. 
%Note that the Ito integral equation whose solution induces the path measure in (OPT), can also be viewed as a randomized ODE by using the classic Doss-Sussmann transformation~\cite{protter} when $d=1$, or by the transformation in Imkuller-Schmalfuss~\cite{} when $d \geq 1$. 

\paragraph{Remark 2.} Observe that these control problem formalisms all assume that the objectives are constrained by the forward state dynamics --  both the mean field and UOCP model the forward dynamics by ODEs, ensemble control by the Fokker-Planck PDE, and drift rate control by an SDE. In contrast, (OPT) is constrained by the feasible set of drift functions, $\mathcal F$, and pushes the stochasticity entirely to the path measure $\pi^x$. This is a key fundamental difference between the control formalism's and the drift optimization formalism, and dictates the method of solution. In Section~\ref{sec:doss-sussmann} of the appendix, we reformulate drift optimization as an optimal control problem. 

\paragraph{Remark 3.} We believe that in many operational and managerial settings, drift optimization, in particular, provides a more natural formalism to use. Indeed, in the control formulations, the well-posedness of the system dynamics requires strict regularity conditions on the control and drift rate functions, which may render the models inappropriate for a given problem setting. On the other hand, we only require the drift functions $F \in \mathcal F$ to be continuous (and could be as ``rough'' as Brownian motion, for instance). In particular, the stochastic optimal control setting, for instance in the drift rate control formalism in Section~\ref{sec:drc}, is really only applicable in `high frequency' settings where the the control and state changes are happening at the same time scale. On the other hand, in Example~\ref{ex:rbm-opt}, for instance, the control and system state `change' on completely different time scales, and drift optimization is a more appropriate mechanistic model.  % Furthermore, in many cases, the control problems are solved by using the dynamic programming principle (DPP), which suffers from the curse of dimensionality.  %Next, we explore the known existing functional stochastic optimization formalisms that use Monte Carlo sampling for solving the problem. 
%\rpcomment{(1) The relationship of these control formalisms to (OPT) is not clearly coming through. (2) ?? appearing in Remark 1 perhaps due to the removal of a section?}

\section{Key Definitions and Notation}~\label{sec:def}
In the definitions that follow, the space $C$ is the \emph{normed} space of $\mathbb{R}^d$-valued continuous functions. For more on convex functionals defined over normed spaces, see~\cite{1997lue}, and~\cite{2006nicper}. %\rpcomment{In (OPT), we said $\mathcal{F}$ is a subset of C, whereas here, we say $\mathcal{F}$ is a subspace of $X$. Do we need $\mathcal{F}$ to be a subsapce?}

\begin{definition}[Linear Functionals] $L: C \to \mathbb{R}$ is called a linear functional on the (real) normed space $C$  if $L(\alpha F) = \alpha L(F), \alpha \in \mathbb{R}$ and $L(F_1 + F_2) = L(F_1) + L(F_2), F_1, F_2 \in C.$ A linear functional $L:C \to \mathbb{R}$ is said to be a \emph{bounded linear functional} if $\|L\| := \sup\left\{|L(F)|: \|F\|=1, F \in C\right\} < \infty.$ It can be shown that $L: C \to \mathbb{R}$ is a bounded linear functional if and only if $L$ is continuous on $C$,  and that continuity of $L$ at any point $F_0 \in C$ implies boundedness of $L$. %\rpcomment{Here we are using $x$ to represent functionals whereas elsewhere $x$ represents an element of $\mathcal{C}$ and is hence $\mathbb{R}^d$-valued. Also, it seems that in all these definitions, $\mathcal{F}$ is assumed to be a normed space. I think we can safely change all $\mathcal{F}$ to $\mathcal{C}$ in these definitions.} %(It is important that $x: \mathcal{F} \to \mathbb{R}$ being bounded does not mean $\sup_{F \in \mathcal{F}} |L(F)| < \infty;$ indeed, it is routinely the case that $\|x\| < \infty$ but $\sup_{F \in \mathcal{F}} |L(F)| = \infty$.)      
\end{definition}

\begin{definition}[Dual Space, Adjoint Space, Conjugate Space]
The space $C^*$ of \emph{bounded} linear functionals on $C$ is called the \emph{dual space} of $C$. %$\mathcal{F}^*$ is sometimes also called the \emph{adjoint space} or the \emph{conjugate space} of $\mathcal{F}$. 
\end{definition}

\begin{definition}[Dual Norm] The operator norm of the functional $T \in C^*$ is called the \emph{dual norm} or \emph{conjugate norm} of $T$: $\|T\|_* := \sup\left\{\frac{|Tz|}{\|z\|}: z \in C, z \neq 0\right\} = \sup\left\{|Tz|: z \in C, \|z\| =1\right\}.$
\end{definition}

\begin{definition}[Right and Left Directional Derivatives]\label{def:dirder} The \emph{right directional derivative} $J'_+(F,v)$ and the \emph{left directional derivative} $J'_{-}(F,v)$ of the functional $J: C \to \mathbb{R}$ at the point $F \in C$, and along direction $v \in \mathcal{C}$, are given by
$J'_+(F,v) := \lim_{t \to 0^+} \frac{1}{t}\left(J(F+tv) - J(F)\right)$ and
$J'_{-}(F,v) := \lim_{t \to 0^{-} } \frac{1}{t}\left(J(F+tv) - J(F)\right).$
\end{definition}

\begin{definition}[G\^{a}teaux and Fr\'{e}chet Differentiability]\label{def:gatfre} The functional $J: C \to \mathbb{R}$ is \emph{G\^{a}teaux differentiable} at $F \in C$ if the following limit exists for each $v \in C$:
\begin{equation}\label{defn:gateaux} S_J(F)(v) := \lim_{t \to 0} \frac{1}{t} \left( J(F+tv) - J(F)\right).\end{equation}
 The functional $J: C \to \mathbb{R}$ is \emph{Fr\'{e}chet differentiable} if the limit in~\eqref{defn:gateaux} holds uniformly in $v$, meaning $$| J(F+v) - (J(F) + S_J(F)(v))| = o(\|v\|), \quad v \in C.$$ From Definition~\ref{def:dirder} and Defintion~\ref{def:gatfre}, we see that Fr\'{e}chet differentiability implies G\^{a}teaux differentiability, which in turn guarantees the existence of  directional derivatives. Also, if $C$ is finite-dimensional and $J$ is Lipschitz in some neighborhood of $F \in C$, then $J$ is Fr\'{e}chet differentiable at $F$ if and only if it is G\^{a}teaux differentiable at $F$.
\end{definition}

\begin{definition}[Subgradient and Subdifferentials of a Convex Functional]\label{def:subgradsubdiff} 
The functional $J: C \to \mathbb{R}$ is \emph{convex} if for any $\alpha \in [0,1]$, it satisfies $J(\alpha F_1 + (1-\alpha)F_2) \leq \alpha J(F_1) + (1-\alpha)J(F_2), \quad \forall F_1,F_2 \in C$. $\tilde{S}_J(F_0) \in C^*$ is called a \emph{subgradient} to $J$ at $F_0 \in C$ if 
\begin{align}\label{eq:subgrad}
  J(F) \geq J(F_0) + \tilde{S}_J(F_0)(F-F_0).  
\end{align}
The set $\partial J(F_0)$ of subgradients to $J$ at $F_0$ is called the \emph{subdifferential} of $J$ at $F_0$. Convex functionals have a subdifferential structure in the sense that if $J:C \to \mathbb{R}$ is convex, then $\partial J(F_0) \neq \emptyset$ for each $F_0 \in C$; conversely,  if $\partial J(F) \neq \emptyset$ for each $F \in C$, then $J$ is necessarily a convex functional.
\end{definition}

\begin{definition}[Mirror Map] \label{def:mirror} Suppose $\bar{\mathcal{D}} \supset C$ and $\mathcal{D} \cap C \neq \emptyset$. A map $\psi: \mathcal{D} \to \mathbb{R}$ is called a \emph{mirror map} if it satisfies the following three conditions: \begin{enumerate} \item $\psi$ is Fr\'{e}chet differentiable and strongly convex in $\mathcal{D}$; \item for each $y \in C^*$, there exists $F \in C$ such that $\nabla \psi (F) = y$; and \item $\lim_{F \to \partial \mathcal{D}} \|\psi(F)\|_* = +\infty,$ \end{enumerate}
\end{definition} where $F \to \partial \mathcal{D}$ is to be understood as $\mbox{dist}(F,\mathcal{D}) := \inf\{\|F-y\|: y \in \mathcal{D}\} \to 0.$

\subsection{Key Assumptions}
We now list key assumptions on the cost functional $\tilde J : C \to \mathbb R$ appearing in~\eqref{opt}, which will be used in the subsequent results. Recall that $\mathcal F$ is a convex, compact subset of $C$. 

% That is, we assume
% 
%We have the following easy lemma.
%\begin{lemma}~\label{ass:diameter}
%    $\mathcal{F}$ has a finite diameter. That is, $diam(\mathcal F) := \sup_{F_1, F_2 \in \mathcal{F}} \|F_1 - F_2\|_\infty < +\infty$.
%\end{lemma}
% \noindent This is a reasonably strong assumption, that is nonetheless satisfied by many problem settings. For example, it is satisfied when the function class $\mathcal{F}$ is parameterized by a compact set.  %\hhcomment {Note that this assumption is crucial for ensuring that the optimal drift estimator is consistent. Indeed, by assuming that the diameter of $\mathcal{F}$ is finite, we have implicitly imposed a finite-dimensionality condition on the feasible set.} %We also believe that it should be possible to relax this condition, at the expense of more complicated computations.
\begin{assumption}~\label{ass:tildeJ}
    The cost functional $\tilde J : C \to \mathbb R$ is G\^{a}teaux differentiable.
\end{assumption}

\begin{assumption}~\label{ass:LipschitzinZ}
    The cost functional $\tilde J$ is $\kappa$-Lipschitz in $z \in C$, i.e., for any $F\in\mathcal F$ and $z,z' \in C$, we have $        |\tilde J(z+F)-\tilde J(z'+F)|\leq \kappa\|z-z'\|_{\infty}.$
\end{assumption}

\begin{assumption}~\label{ass:LipschitzinF}
    For a fixed path $z\in C\sim\pi^x$,  $|\tilde{J}(z+F_1) - \tilde{J}(z+F_2)| \leq K_{z} \|F_1-F_2\|_\infty,$ where $K_{z} > 0$ for every  $F_1, F_2 \in \mathcal F$, and {$\mathbb E[K_{z}^p] < +\infty$} for some $2 \leq p < +\infty$.
    %The cost functional $\tilde{J}:C \to \mathbb R$ is $K_{z}$-Lipschitz on $\mathcal F$
\end{assumption}
Observe that Assumption~\ref{ass:LipschitzinF} does not follow from Assumption~\ref{ass:LipschitzinZ}. While the latter concerns fixed ``shift'' functions, the former explicitly concerns the effects of two different shifts on the same sample path $z$.
%\hl{HH: add explanation here for these Lipschitz assumptions.}
\begin{assumption}
    The sample path objective $\tilde J : C \to \mathbb R$ is convex. 
\end{assumption}

For simplicity, we adopt the following assumption throughout the paper.
\begin{assumption}\label{ass:fix_mu0}
    The measure $\mu_0$ is a Dirac measure concentrated at $x \in \mathbb R^d$.
\end{assumption}
This will entail no loss of generality in our main results (which concern the SAA primarily), but will help keep the narrative and notation simple.

A straightforward implication of this condition is that the objective in expectation, $J(F) = \int_C \tilde J \circ \Gamma(z+F) \pi^x(d\,z)$, is also convex. %\rpcomment{$J$ as expressed here is not the same as $J$ expressed in (OPT)} \hl{HH:moved Assumption 10 up above. Takes care of this issue.}

\begin{assumption}\label{ass:GammaLipschitz}
    The regulator function $\Gamma:C\rightarrow C$ is $L_\Gamma$-Lipschitz on $C$. %\rpcomment{if you are spell}
\end{assumption}
This assumption is easily satisfied by the Skorokhod regulator which is $2$-Lipschitz continuous in the space $C$. As a direct consequence of Assumption~\ref{ass:LipschitzinZ} and Assumption~\ref{ass:GammaLipschitz}, the composed operator $\tilde{J}\circ\Gamma(z + F))$ is $\kappa L_\Gamma$ Lipschitz in $z$ for a fixed $F$. Similarly, by Assumption~\ref{ass:LipschitzinF} and Assumption~\ref{ass:GammaLipschitz}, $\tilde{J}\circ\Gamma(z + F))$ is $K_{z} L_\Gamma$ Lipschitz in $F$ for a fixed sample path $z$.

\begin{assumption}~\label{ass:gamma}
    The regulator function $\Gamma$ is G\^{a}teaux differentiable.
\end{assumption}

%Note that, in general, there is no natural extension of Rademacher's theorem~\cite{} to infinite-dimensional spaces. Nonetheless,~\cite{bogachev-mayer-wolf}, shows that if $\mi^x$ is the Wiener measure (or equivalent to it), the directional derivative as assumed above exists.

\begin{assumption}~\label{ass:composed} The cost functional $\tilde J : C \to \mathbb{R}$ is integrable, that is,
\(
    \int_C \left|\tilde{J}(z)\right| d\pi^x(z) < +\infty.
\)\end{assumption}

\begin{assumption}~\label{ass:covering}
The covering number of the space $\mathcal F$ satisfies $\log N(\epsilon,\mathcal{F},\|\cdot\|_{\infty}) \leq \epsilon^{-1/\alpha}$ for some $\alpha > 1$ and $\epsilon > 0$.
\end{assumption}

\section{Unbiased Derivative Estimator}~\label{sec:dir-der}
As observed in Section~\ref{sec:Intro}, our strategy to solve the problem in~\eqref{opt} involves multiple levels of approximation, the first of which approximates the integral in~\eqref{opt} through Monte Carlo:
\begin{align}~\label{saafunctional}~\tag{MC-OPT}
    \text{minimize}&~J_N(F):=\frac{1}{N} \sum_{i=1}^N \tilde J \circ \Gamma (Z_i + F),\\
    \text{subject to:}&~F \in \mathcal F, \nonumber
\end{align}
where $N$ is the sample size and $\{Z_i : i = 1,\ldots, N\}$ are i.i.d. samples from the distribution $\pi^x$. The question we consider in this section pertains to the directional derivative of the pathwise cost functional appearing in~\eqref{saafunctional}. 

We demonstrate that the directional derivative of the pathwise cost functional is an unbiased estimator of the expected cost functional $J$. We start with the general setting in Assumption~\ref{ass:tildeJ} and Assumption~\ref{ass:gamma}. Let $Y^F = Z+F$, where $F \in \mathcal F$ and $Z \sim \pi^x$ with $x\in\mathbb R^d$. Let $\pi_F^x$ represent the path measure corresponding to $Y^F$. For notational simplicity, we will use $Y$ in place of $Y^F$. %Then, it is straightforward to see that
For any path $y$ sampled from $\pi_F^x$, let $D_u\tilde{J}\circ\Gamma(y)~(\equiv (D_u \tilde J \circ \Gamma) (y))$ represent the directional derivative of $\tilde{J}\circ\Gamma(y)$ with respect to $y$ along the direction $u\in C$.
\begin{lemma}~\label{lem:chain-rule}
     Suppose $\tilde{J}$ satisfies Assumption~\ref{ass:tildeJ}, then  $D_u\tilde{J}\circ\Gamma(y)$ exists and satisfies
     %\begin{align*}
     \(
         D_u\tilde{J}\circ\Gamma\,(y) = \left(D_{D_u\Gamma(y)}\tilde{J} \right)\,(\Gamma(y)),
    \)
     %\end{align*}
     where $D_u\Gamma(y)$ is the directional derivative of the regulator map.
\end{lemma}
This is a mere restatement of the chain rule for G\^{a}teaux differentiable functions and requires no further elaboration. The next theorem shows that the pathwise derivative is an unbiased estimator.

\begin{theorem}~\label{thm:objective-derivative}
    Fix $x \in \mathbb R^d$. The pathwise directional derivative of the pathwise cost functional is an unbiased estimator of the directional derivative of the expected cost functional. That is
    $$\mathbb E\left[ D_u \tilde J \circ \Gamma (Y) \right] = D_u J(F).$$
    %$$\mathbb \int D_u\tilde{J}\circ\Gamma(y) d\pi^x_F(y)= D_u \int \tilde{J}\circ\Gamma(y) d\pi^x_F(y)$$
\end{theorem}

\begin{proof}{Proof}
    First, observe that $\frac{1}{\epsilon} \left(\tilde{J}\circ\Gamma(y + \epsilon u) - \tilde{J}\circ\Gamma(y)\right)$ is integrable:
\begin{align*}
    &\int \frac{1}{\epsilon} \left(\tilde{J}\circ\Gamma(y + \epsilon u) - \tilde{J}\circ\Gamma(y)\right) d\pi_F^x(y)
    \\
    &\qquad\leq\int \frac{K_y L_{\Gamma}}{\epsilon} \epsilon\left\|u\right\| d\pi_F^x(y)\\
    &\qquad=K_y L_{\Gamma}\left\| u\right\|,
\end{align*}
where the first inequality follows from Assumption~\ref{ass:GammaLipschitz}. Thus, we conclude that
\begin{align*}
    &\mathbb E[D_u\tilde{J}\circ\Gamma(Y)]\\
   &\quad=\int \lim_{\epsilon\rightarrow 0} \frac{1}{\epsilon} \left(\tilde{J}\circ\Gamma(y + \epsilon u) - \tilde{J}\circ\Gamma(y)\right) d\pi_F^x(y)\\
   &\quad= \lim_{\epsilon\rightarrow 0} \frac{1}{\epsilon}\int\left(\tilde{J}\circ\Gamma(y + \epsilon u) - \tilde{J}\circ\Gamma(y)\right) d\pi_F^x(y)\\
   &\quad=D_u\mathbb E[\tilde{J}\circ\Gamma(Y)] = D_u J(F),
\end{align*}
where the penultimate equality is due to the Lebesgue dominated convergence theorem.
\end{proof}

Theorem~\ref{thm:objective-derivative} immediately suggests that we can use sample average approximation (SAA) to solve~(OPT), by approximating the expectation using i.i.d. Monte Carlo samples $\{Y_1,\cdots,Y_N\}$ drawn from the path measure~$\pi_F^x$, yielding the sample average objective
 \begin{align}\tag{MC-OPT}
     \min_{F \in\mathcal{F}} \frac{1}{N}\sum_{i=1}^N \tilde J \circ \Gamma (Y_i),
 \end{align}
and the sample average pathwise derivative
\begin{align*}
    {D}_u \hat{\mathbb E}[\tilde{J}\circ\Gamma(Y)] = \frac{1}{N}\sum_{i=1}^N D_u\tilde{J}\circ\Gamma(Y_i),
\end{align*}
 where $\hat{\mathbb E}$ is an expectation with respect to the empirical measure. The sample average above is an unbiased estimator of the directional derivative $D_u J(F)$. 
 We term this the {\it functional SAA}, to distinguish from the classical finite-dimensional SAA. 

\subsection{Skorkhod Regulated Processes}
Using the functional SAA in practice will require the computation of the (G\^{a}teaux) directional derivative of the regulated process $D_u \Gamma (\cdot)$. In this section, we specialize to the case of a Skorokhod regulated process~\cite[Chapter 7]{chen2001fundamentals},\cite[Ch. 14]{Whitt_2002_supplement} and derive the directional derivative $D_u \Gamma$. The Skorokhod regulated process is defined as 
\begin{align}~\label{eq:rbm}
	\Gamma\left(Y\right) (t) = Y(t) + L(t)
\end{align}
where $L$ is the multidimensional Skorokhod regulator process, defined as the fixed point $r$ of the functional given by $\pi(r)(t) = \sup_{0\leq s \leq t}\{-(Y(s) + Q r(s))_+\}$, and $Q$ is a column sub-stochastic matrix; see \citep[Sec. 14.2]{Whitt_2002} for instance. Here, for simplicity, we assume that $Q = 0$, which yields the closed form
\begin{align*}
   L(t) &= \sup_{0 \leq s\leq t}\{-Y(s)\}_+\\
   &= \Psi(Y)(t),
\end{align*}
where $\Psi:Y\in C\mapsto L\in C$ is the so-called Skorokhod regulator map.
% In particular, this requires demonstrating that the Skorokhod regulator is G\^{a}teaux differentiable, satisfying Assumption~\ref{ass:gamma}. 
% %Recall from~\eqref{eq:rbm} that the driving process $Z$ is defined as $Y(t)= x + B(t)$, where $B$ is a standard Brownian motion. 
% Let $\Psi : C \to C$ be the Skorokhod regulator map~\citep{skorokhod1961stochastic}. Fix $Y \in C$, so that $\Psi(Y) \equiv L$, where $L$ is the regulator process as defined in Section~\ref{sec:setting}. Recall that $\Psi(Y)(t):=\sup_{0\leq s\leq t}\{(-Y_s)_+\}=L(t)$. 
% First, we establish that the Skorokhod regulator map is G\^{a}teaux differentiable at each time point $t \in [0,T]$. Recall 

Next, to show that the Skorokhod regulator map is G\^{a}teaux differentiable at each time point $t \in [0,T]$, we recall Danskin's theorem,
\begin{lemma}[Danskin's Theorem~\citep{BERNHARD19951163}]\label{lem:5.0}
    Let U and V be subsets of a Banach space  $\mathcal{U}$ and a topological space $\mathcal{V}$. Let $M:U\times V\to\mathbb{R}$. Denote by $D_1M(u,v;h)$ the directional derivative of $u\mapsto M(u,v)$ in the direction $h\in U$. Let $\Bar{M}(u)=\sup_{v\in V}M(u,v)$ and $\Phi_V(u) = \{v\in V: M(u,v)=\Bar{M}(u)\}$. Under following conditions:
    \begin{enumerate}
        \item V is compact.
        \item The mapping $(\delta,v)\mapsto M(u+\delta h, v)$ is upper semi-continuous.
        \item $\forall v\in V$ and $\forall \delta$ in a right neighborhood of 0, there exists a bounded directional derivative
        $$D_1M(u+\delta h,v;h)=\lim_{\epsilon\to 0^+}\frac{1}{\epsilon}[M(u+(\delta +\epsilon)h,v)-M(u+\delta h,v)].$$
        \item The map $(\delta ,v)\mapsto D_1M(u+\delta h,v;h)$ is upper semi-continuous at (0,v).
    \end{enumerate}
    Then, $\Bar{M}$ has a directional derivative at u in the direction h given by
    $D_h\Bar{M}(u)=\max_{u\in\Phi_V(u)}D_1M(u,v;h).$
\end{lemma}

Recall that we seek the pathwise direciotnal derivative of the objective functional, $D_u\tilde{J}$. Lemma~\ref{lem:chain-rule} shows that it can be written in terms of the pathwise directional derivative of the Skorokhod regulator, $D_u\Gamma$. Next, we derive the explicit expression for $D_u\Gamma$ applying Danskin's Theorem, which in turn completes the derivation for $D_u\tilde{J}$.

\begin{proposition}[Pathwise Directional Derivative]\label{prop:Sko-pathwise}
The directional derivative of the regulated process $\Gamma(y)(t)$ exists pointwise for each $t \in [0,T]$. Furthermore, with probability 1, there are two cases for the directional derivative of $Z\mapsto\Gamma(y)(t)$ along direction $u\in C$ which is given by
\begin{align*}
&D_u\Gamma(y)(t)\\
&=\begin{cases}
u(t) + \sup_{s\in \Phi_t(y)} \{-u(s)\}, &\text{case 1}.\\
u(t), &\text{case 2}.
\end{cases}
\end{align*}
where 
\begin{align*}
    &\text{Case 1:} \\
    &\quad y(s): 0 < s \leq t \text{ goes below } 0\text{ at least once, or} \\
    &\quad y(0) = 0, u(0) < 0 \text{, and } y(t) > 0, \forall s \in (0,t]. \\
    &\text{Case 2:} \\
    &\quad y(t) > 0, \forall s \in [0,t], \text{ or} \\
    &\quad y(0) = 0, u(0) \geq 0\text{, and } y(t) > 0, \forall s \in (0,t]. 
\end{align*}
Here, $u\in\mathcal{F}$ and $\Phi_t(y):= \{0\leq s \leq t : \Gamma(y)(t) = y(t)-y(s)\}$ is the set of all time points $s\leq t$ where the function $y$ attains its infima.
\end{proposition}

%\zzcomment{How do we transition to the next Lemma by M\&M.}

\begin{lemma}\label{lem:5.2}
The Skorokhod regulator map is pointwise Frechet differentiable. That is, for each $t\in[0,T]$, if the process $y(s):0< s\leq t$ either attains negative values at least once or $y(0)=0$, $u(0)< 0$, $y(t)>0\,\forall \,s\in(0,t]$, we have
$$|\Psi(y+u)(t)-\Psi(y)(t)-\sup_{s\in\Phi_t(y)}\{-u(s)\}|=o(||u||).$$
where $u\in\mathcal{F}$ and $\Phi_t(y) := \{0\leq s \leq t : \Psi(y)(t) = -y(s)\}$ is the set of all time points $s$ up to $t$ where the function $y$ attains its infima.
If $y(t)>0\$$ for all $s\in[0,t]$, or if $y(0)=0$, $u(0)\geq 0$, $y(t)>0$ for all $s\in(0,t]$, then we obtain
$$|\Psi(y+u)(t)-\Psi(y)(t)|=o(||u||).$$
\end{lemma}

Lemma~\ref{lem:5.2} recovers the `Fundamental Lemma' in~\cite{MM}. In the optimization setting of interest, it will be useful to establish the somewhat stronger result in Lemma~\ref{lem:5.3} next.

\begin{lemma}\label{lem:5.3}
    $\frac{\Gamma(y+\epsilon u)-\Gamma(y)}{\epsilon}\to D_u\Gamma(y)$ in $L^p$.
\end{lemma}

\begin{proof}{Proof}%[Proof of Lemma~\ref{lem:5.3}]
    Given that $\frac{\Gamma(y+\epsilon u)(t)-\Gamma(y)(t)}{\epsilon}\to D_u\Gamma(y)(t)$ pointwise for all $t\in[0,T]$. We aim to show that this convergence holds in $L^p$. To see this, observe that by Lemma~\ref{lem:5.2}, we have
\begin{align*}
    &\left| \frac{\Gamma(y+\epsilon u)(t)-\Gamma(y)(t)}{\epsilon}- D_u\Gamma(y)(t)\right|\\
    &\quad=\frac{1}{\epsilon}\left|\Gamma(y+\epsilon u)(t)-\Gamma(y)(t)- D_{\epsilon u}\Gamma(y)(t)\right|\\
    &\quad=\frac{1}{\epsilon}\left|\Psi(y+\epsilon u)(t)-\Psi(y)(t)- D_{\epsilon u}\Psi(y)(t)\right|\\
    &\quad=\frac{1}{\epsilon}o(||\epsilon u||)=o(|| u||).
\end{align*}  

Since $|| u||$ is bounded on compact support, we apply the Dominated Convergence Theorem to conclude that
\begin{align*}
    \frac{\Gamma(y+\epsilon u)-\Gamma(y)}{\epsilon}\to D_u\Gamma(y)\quad\text{in } L^p.
\end{align*}
% :
% \begin{equation*}
% \begin{aligned}
% &\lim_{\epsilon\to 0}\int_{0}^{T}&\Bigg|\frac{\Gamma(y+\epsilon u)(t)-\Gamma(y)(t)}{\epsilon}\\
% &&-D_u\Gamma(y)(t)\Bigg|^p dt\\
% =&\int_{0}^{T}\lim_{\epsilon\to 0}&\Bigg| \frac{\Gamma(y+\epsilon u)(t)-\Gamma(y)(t)}{\epsilon}\\
% &&- D_u\Gamma(Y)(t)\Bigg|^p dt\\
% =~&0.
% \end{aligned}
% \end{equation*}
 
\end{proof}

\paragraph{Remark.} The directional derivative of the Skorokhod regulator was derived in~\cite{MM,Whitt_2002_supplement} in the scalar setting, we provide a new proof that leverages Danskin's theorem from convex analysis. To our knowledge, this connection has not been made before in the literature. Note that the directional derivative for the full multidimensional Skorokhod regulator has been derived in~\cite{MR}, and is considerably more complicated than the scalar setting. Indeed,~\cite{MR} are able to derive the object pointwise (for each $t \in [0,T]$) in greater generality and over the space of right continuous functions that have left limits (i.e., the so-called Skorokhod space). The complication arises partly because of the definition of the Skorokhod regulator process $L$ as the unique fixed point of the functional $\pi(r)(t) = \sup_{0 \leq s \leq t} \{-(Y(s) + Q r(s))_+\}$. Our proof above relies on the condition that $Q = 0$ and the paths lie in $C$, a subset of the Skorokhod space. However, we believe that Danskin's theorem remains valid even for $Q \neq 0$ and over the full Skorokhod space. This follows from the fact that the Skorokhod space forms a Banach space when equipped with the uniform (sup) norm. To keep the discussion tightly focused on SAA for generally regulated stochastic models (and not just Skorokhod regulated processes) we will not further elaborate this result and leave the details to a separate paper.
\section{{Equiconvergence and Consistency of Functional SAA}}~\label{sec:equi}
%A natural question at this point is whether the optimal value of the SAA problem (MC-OPT)  is `consistent' with the the `true' optimal value in (OPT). 
In the previous section, we established that the pathwise directional derivative can be used to construct an unbiased estimator of the true objective's gradient. This result  legitimizes a \emph{functional SAA} scheme defined over the full function subspace $\mathcal{F}$. Recall the original problem:
\begin{align*}
\tag{OPT}
 &~\text{min:}~J(F) := \underset{C}\int \tilde J\circ\Gamma(z+F)\pi^x(dz),\\
\nonumber&\text{subject to:}~
%&\qquad\qquad  X^F_t(z) = x + F(t) + \int_0^t \sigma(s) dB_s(z) + L(t) ~\forall~t \geq 0\\
F \in \mathcal F.
\end{align*}
Then the (intermediate) functional SAA problem is
\begin{align}~\tag{MC-OPT}
    \text{minimize}&~\frac{1}{N} \sum_{i=1}^N \tilde J \circ \Gamma (Z_i + F),\\
    \text{subject to:}&~F \in \mathcal F, \nonumber
\end{align}
where $N$ is the sample size and $\{Z_i : i = 1,\ldots, N\}$ are i.i.d. samples from $\pi^x$. In this section, we establish consistency for the functional SAA in (MC-OPT) as $N\rightarrow\infty$. We show that the consistency can be achieved without any discretization when optimizing over the function space $\mathcal {F} $ and sampling from the path measure $\pi^x$. For simplicity, throughout this section let us assume that $\Gamma$ is the identity map. We will subsequently observe that the forthcoming results extend to the regulated case.
%\textcolor{red}{Note that the diffusion coefficient = 1 WLOG.}
We now state the central equiconvergence theorem which shows that the SAA objective is uniformly close to the true objective over the space $F\in\mathcal F$ with high probability. Recall from Assumption~\ref{ass:covering} that $\log N(\epsilon, \mathcal{F}, \|\cdot\|_\infty)$ represents the covering number of the feasible space $\mathcal{F}$ under the sup-norm.
\sloppy
\begin{theorem}~\label{thm:equiconv2}
%Suppose that $\log N(\epsilon, \mathcal F, \|\cdot\|_{\infty}) \leq \epsilon^{-1/\alpha}$ for $\alpha \geq 1$ and $\epsilon > 0$.
Suppose the cost function $\tilde J$ satisfies Assumption~\ref{ass:tildeJ},  Assumption~\ref{ass:LipschitzinZ}, Assumption~\ref{ass:LipschitzinF} (for some $1 \leq p < +\infty$) and Assumption~\ref{ass:covering}. Let $\mathbf{Z} = (Z_1,\cdots,Z_N)$ be an i.i.d. sample drawn from $\pi^x$. Then, for any $\delta > 0$ and some $1 \leq p < +\infty$, with probability at least $1-\delta$, for any $F \in \mathcal{F}$ we have
\begin{align*} 
    J(F) &\leq \frac{1}{N} \sum_{i=1}^N \tilde J(Z_i+F) + O\left( \sqrt{\frac{\log(1/\delta)}{N}} \right).
\end{align*}
\end{theorem}
There are two main ingredients to prove Theorem~\ref{thm:equiconv2}. We start by bounding the Gaussian complexity of the SAA (see \eqref{eq:gc-def} below), as a direct consequence of Proposition~\ref{thm:gauss-complex}. This in turn feeds into a high probability concentration bound of the SAA around the true objective (Proposition~\ref{thm:equiconv}). All the proofs are in the Appendix~\ref{sec:equi-proofs}.

We now supply some supporting lemmas in preparation for Proposition~\ref{thm:gauss-complex}. Consider the $\mathbb R^N$-valued random field $\mathcal{G}_\cdot(\cdot)$, defined as
\(
    F \mapsto \mathcal{G}_F(\mathbf{Z}) := \left( \tilde{J}(Z_1+F), \cdots \tilde{J}(Z_N+F)\right).
\)
Let $C^N = \underset{N~\text{times}}{\underbrace{C\times\cdots\times C}}$, and for each $ \boldsymbol\xi  \in C^N$ define the set $\mathcal{B} \equiv \mathcal{B}( \boldsymbol\xi ) := \{\mathcal{G}_F( \boldsymbol\xi ) : F \in \mathcal{F}\} \subseteq \mathbb R^N$, and the pseudometric $d $ for each $(a,b) \in \mathcal{B} \times \mathcal{B} : \to [0,\infty)$
\begin{align}~\label{eq:pmetric}
    d(a,b) = \frac{1}{\sqrt{N}} \|K_{ \boldsymbol\xi }\|_{p} \|F_a - F_b\|_{\infty},
\end{align}
where $F_a,F_b \in \mathcal F$ correspond to $a,b$ (respectively) through the map $\mathcal{G}_\cdot$, for any $v \in \mathbb R^N$, $\|v\|_p := \left(\sum_{i=1}^N |v_i|^p\right)^{1/p}$ and $K_{ \boldsymbol\xi } = (K_{\xi_1},\cdots,K_{\xi_N})$ is the vector of  Lipschitz constants in Assumption~\ref{ass:LipschitzinF}.

Next, let $\{\mathcal{Y}_F(\mathbf Z) : F \in \mathcal{F}\}$ be the real-valued random field defined as
\begin{align}
    \mathcal{Y}_F(\mathbf Z) := \frac{1}{\sqrt{N}} \sum_{i=1}^N g_i \tilde{J}(Z_i+F),
\end{align}
where $g$ is a $N$-dimensional standard Gaussian random vector, as before. %It is straightforward to see from Lemma~\ref{lem:sep-field} that $\{\mathcal{Y}_F(\mathbf Z) : F \in \mathcal F\}$ is separable as well. 

The next lemma shows that $\{\mathcal{Y}_F(\mathbf Z) : F \in \mathcal F\}$ satisfies a sub-Gaussian concentration inequality, when conditioned on $\mathbf Z$.

\begin{lemma}~\label{lem:sub-gauss}
For any $F, G \in \mathcal{F}$ such that $F\neq G$ we have 
%\begin{align*}
\(
    \mathbb P \left( \left| \mathcal{Y}_{F}(\mathbf{Z}) - \mathcal{Y}_G(\mathbf{Z})\right| > u \right | \mathbf Z =  \boldsymbol\xi ) \leq 2 \exp \left( - \frac{u^2}{2 L^2 d(\mathcal{G}_F( \boldsymbol\xi ),\mathcal{G}_G( \boldsymbol\xi ))^2} \right),
\)
%\end{align*}
where 
%\begin{align} 
    %\label{eq:L-constant}
 \(   L = \underset{y \in \mathbb R^N : \|y\|_{2} = 1}{\sup} \|y\|_{q} = 1 ~\text{for}~q \geq 2.\)
    % \nonumber
    % &=\begin{cases}
    %     N^{1/q - 1/2} &~\text{if}~q < 2\\
    %     1 &~\text{if}~q \geq 2.
    % \end{cases}
%\end{align}
\end{lemma}

Equiconvegrence property also depends on the geometric spread of the feasible set $\mathcal{F}$. The following result provides the required diameter bound. Recall that $\mathcal F$ is a convex, compact subset of $C$. 
We have the following easy lemma.
\begin{lemma}~\label{ass:diameter}
    $\mathcal{F}$ has a finite diameter. That is, $diam(\mathcal F) := \sup_{F_1, F_2 \in \mathcal{F}} \|F_1 - F_2\|_\infty < +\infty$.
\end{lemma}

\begin{proposition}~\label{thm:gauss-complex}
Suppose Assumption~\ref{ass:tildeJ} and Assumption~\ref{ass:covering} hold, then, for any $F_0 \in \mathcal F$, there exists a constant $0 < \mathtt{C} < +\infty$ such that
\begin{align}~\label{eq:gauss-complex1}
    &\mathbb E_{g} \left[ \sup_{F \in \mathcal F} |\mathcal{Y}_F(\mathbf Z) - \mathcal{Y}_{F_0}(\mathbf Z)| \bigg | \mathbf Z =  \boldsymbol\xi  \right] \\\nonumber&\qquad\leq \frac{ \mathtt{C} \|K_{ \boldsymbol\xi }\|_{p}}{\sqrt{N}} \left(\frac{1}{2} diam(\mathcal F) \right)^{\frac{\alpha-1}{\alpha}}.%C \int_0^{D/2} \sqrt{ \log N(\epsilon,\mathcal{B},d)} d\epsilon,
\end{align}
\end{proposition}

\newcommand{\dsg}{\Delta_{\text{SG}}}
Recall that the sub-Gaussian diameter for a metric probability space $(\mathcal{X}, d,\pi)$ with metric $d$ and measure $\pi$ is defined as $\dsg^2(\mathcal{X}) := \sigma^*(\mathcal{H})$ where $\sigma^*(\mathcal{H})$ is the smallest $\sigma$ that satisfies 
%\begin{align}
\(
    \mathbb E \left[e^{\lambda \mathcal{H}} \right]\leq e^{\sigma^2\lambda^2/2},\,\lambda\in\mathbb R,
\)
%\end{align}
$\mathcal{H}:= \epsilon d(X,X')$ is the symmetrized distance on the metric space $\mathcal{X}$, $\epsilon = \pm 1$ with probability 1/2 and $X,X'$ are $\mathcal{X}$-valued random variables with measure $\pi$. 

Next, we prove equiconvergence of the SAA objective over the feasible space $\mathcal{F}$ via a sub-Gaussian McDiarmid-type inequality.

\begin{proposition}~\label{thm:equiconv}
    Let $\mathbf{Z} = (Z_1,\cdots,Z_N)$ be $N$ i.i.d. random variables with measure $\pi^x$. Suppose the cost function satisfies Assumption~\ref{ass:LipschitzinZ}. Suppose that the metric probability space $(C,\|\cdot\|_\infty,\pi^x)$ satisfies $\dsg(C) < +\infty$. Then for any $F\in\mathcal F$ and  $\delta>0$, with probability at least $1-\delta$, we have 
    \begin{equation}\label{eq:sample-complex1}
       \begin{aligned}
        J(F) &\leq \frac{1}{N} \sum_{i=1}^N \tilde J(Z_i+F)\\
        &\qquad+\mathbb E \left[\sup_{G \in \mathcal{F}} \left \{J(G) - \frac{1}{N} \sum_{i=1}^N \tilde J(Z_i+G) \right\} \right]\\
        &\qquad+\left(\frac{2\kappa^2\dsg^2(C) \log(1/\delta)}{N}\right)^{1/2}.
    \end{aligned}
    \end{equation}
\end{proposition}
\noindent{\it Remark:} We note that the assumption that $\dsg(C) < +\infty$ is reasonable -- for instance, it is satisfied in the case where $\pi^x$ is the Wiener measure.

So far we have stated all supporting results to prove the main theorem.
\begin{proof}{Proof of Theorem~\ref{thm:equiconv2}}
From~\eqref{eq:gauss-complex1} in  Proposition~\ref{thm:gauss-complex}, we have
\begin{align}\label{Y_F-bound}
    &\nonumber\frac{ \mathtt{C} \|K_{ \boldsymbol\xi }\|_{p}}{\sqrt{N}} \left(\frac{1}{2} diam(\mathcal F) \right)^{\frac{\alpha-1}{\alpha}}
    \\&\nonumber\qquad\geq \mathbb E_{g} \left[ \sup_{F \in \mathcal F} |\mathcal{Y}_F(\mathbf Z) - \mathcal{Y}_{F_0}(\mathbf Z)| \bigg | \mathbf Z =  \boldsymbol\xi  \right]
    \\&\nonumber\qquad\geq \mathbb E_{g} \left[ \sup_{F \in \mathcal F}\{\mathcal{Y}_F(\mathbf Z)\} - \mathcal{Y}_{F_0}(\mathbf Z) \bigg | \mathbf Z =  \boldsymbol\xi  \right]
    \\&\qquad = \mathbb E_{g} \left[ \sup_{F \in \mathcal F}\{\mathcal{Y}_F(\mathbf Z)\} \bigg | \mathbf Z =  \boldsymbol\xi  \right].
\end{align}
Recall that the Gaussian complexity $\mathcal{R}_N(\mathcal F)$ of the SAA over the function space $\mathcal{F}$ is defined as
\begin{align}~\label{eq:gc-def}
    \mathcal{R}_N(\mathcal F) := \mathbb E_{g,\pi^x}\left[ \sup_{F \in \mathcal{F}} \left\{ \frac{1}{N} \sum_{i=1}^N g_i \tilde{J}(Z_i+F) \right\}\right],
\end{align}
where the expectation is taken with respect to the Gaussian random vector $g\sim \mathcal{N}(0,I_{N\times N})$, which is independent of the i.i.d. samples $\mathbf{Z} := (Z_1, \cdots, Z_N)$ drawn from $\pi^x$.
It follows from \eqref{Y_F-bound} and \eqref{eq:gc-def} that $\mathcal{R}_N(\mathcal F)$ satisfies:
\begin{align*}
    \mathcal{R}_N(\mathcal F) &:= \mathbb E_{g,\pi^x}\left[ \sup_{F \in \mathcal{F}} \left\{ \frac{1}{N} \sum_{i=1}^N g_i \tilde{J}(Z_i+F) \right\}\right]
    \\&\quad = \mathbb E_{g,\pi^x}\left[ \sup_{F \in \mathcal{F}} \frac{1}{\sqrt{N}}\mathcal{Y}_F(\mathbf Z)\right]
    \\&\quad = \mathbb E_{\pi^x}\left[\frac{1}{\sqrt{N}}\mathbb E_{g} \left[ \sup_{F \in \mathcal F}\{\mathcal{Y}_F(\mathbf Z)\} \bigg | \mathbf Z =  \boldsymbol\xi  \right]\right]
    \\&\quad \leq \mathbb E_{\pi^x}\left[\frac{ \mathtt{C} \|K_{ \boldsymbol\xi }\|_{p}}{N} \left(\frac{1}{2} diam(\mathcal F) \right)^{\frac{\alpha-1}{\alpha}}\right].
\end{align*}
Evaluating the above expectation, we get:
\begin{align}~\label{eq:gc}
    \mathcal{R}_N(\mathcal{F})& < \frac{\mathtt C \, \mathbb E \left[ \|K_{\mathbf Z}\|_{p} \right]}{N}\left(\frac{1}{2} diam(\mathcal F) \right)^{\frac{\alpha-1}{\alpha}}
    \\ &\nonumber\quad = O\left(\frac{1}{N}\right),
\end{align}
provided $\mathbb E_{\pi^x}\left[ \|K_{\mathbf Z}\|_{p}\right] < +\infty$; this is a consequence of Assumption~\ref{ass:LipschitzinF}. By standard considerations (see~\cite{bartlett2002rademacher} for instance), it can be shown that
%\begin{align}
%\nonumber
\(
    \mathbb E\left[ \sup_{F \in \mathcal{F}} \left\{ J(F) - \frac{1}{N} \sum_{i=1}^N \tilde{J}(Z_i+F) \right\} \right] \leq 2 \mathcal{R}_N(\mathcal F).
\) 
Therefore, from~\eqref{eq:sample-complex1} in Proposition~\ref{thm:equiconv}, we have 
\begin{align*} 
    J(F) &\leq \frac{1}{N} \sum_{i=1}^N \tilde J(Z_i+F) + 2 \mathcal{R}_N(\mathcal{F})\\&\qquad + O\left( \sqrt{\frac{\log(1/\delta)}{N}} \right),
\end{align*}
with $\mathcal{R}_N(\mathcal F) = O\left(\frac{1}{N}\right)$, dominated by the last term $O\left( \sqrt{\frac{\log(1/\delta)}{N}} \right)$.
 \end{proof}
Next, recall that the composed functional $\tilde J \circ \Gamma$ is $\kappa L_{\Gamma}$-Lipschitz. Consequently, the consistency result proved in Theorem~\ref{thm:equiconv2} holds for the composed functional as well. 

The equiconvergence result also guarantees the consistency of the optimizer. To see this, observe that as an immediate consequence of Theorem~\ref{thm:equiconv2}, we have~%~\cite[Theorem 9]{pasupathy_saasurvey}
\(
     | J^* - \tilde{J}_N^*| \stackrel{P}{\to} 0~\text{as}~N \to \infty,
\)
where $J^* := \inf_{F \in \mathcal{F}} J(F)$ and $\tilde J_N^* := \inf_{F \in \mathcal{F}} \frac{1}{N} \sum_{i=1}^N \tilde{J}\circ\Gamma(Z_i+F)$. Furthermore, let $\Pi_n^* := \arg\inf_{F \in \mathcal{F}} \frac{1}{N} \sum_{i=1}^N \tilde{J}\circ\Gamma(Z_i+F)$ and $\nu^* := \arg\inf_{F \in \mathcal{F}} J(F)$. Consider two scenarios, $J^*>\tilde J_N^*$ and $J^*<\tilde J_N^*$. In the former case, $|J^*-\tilde J_N^*|<|J(\Pi_n^* )-\tilde J_N^*|$. In the latter case,  $|J^*-\tilde J_N^*|<|\frac{1}{N} \sum_{i=1}^N \tilde{J}\circ\Gamma(Z_i+\nu^*)-J^*|$. Therefore
\begin{align*}
 &|J^*-\tilde J_N^*|\\
    &\leq\max\Bigg\{ \left|J(\Pi_n^* )-\tilde J_N^* \right|,\\
    &\quad\left|\frac{1}{N} \sum_{i=1}^N \tilde{J}\circ\Gamma(Z_i+\nu^*)-J^* \right|\Bigg\}\\
    &\leq |J(\Pi_n^* )-\tilde J_N^*|+\left|\frac{1}{N} \sum_{i=1}^N \tilde{J}\circ\Gamma(Z_i+\nu^*)-J^* \right|
    \stackrel{P}{\to} 0
\end{align*}
$~\text{as}~N \to \infty$ by Theorem~\ref{thm:equiconv2}.

\section{{Rate of Convergence of Finite SAA}}~\label{sec:rate}
%In this section we consider a further generalized solution estimator to (OPT) by assuming access to a numerical integrator for solving the SDE~\eqref{eq:lsde} with weak convergence order $\beta > 0$. 
Thus far we have assumed that the measure $\pi^x$ can be sampled without approximation. Furthermore, (MC-OPT) assumes that optimization can be performed over the function space $\mathcal F$. This is a purely theoretical construct, of course. In the next section we relax these conditions to propose an implementable method.

In the spirit of~\citep{stuart2010inverse}, to \textit{``...avoid discretization until the last possible moment...''}, our practicable finite SAA uses the pathwise directional derivative developed above together with mirror descent over a finite dimensional subspace of $\mathcal{F}$, to approximately solve the functional SAA problem. Our main theorem below provides a rate of convergence of the finite SAA solution to that of the functional SAA. We start by describing the various `discretizations' that are necessary for a practicable method.

\subsection{Feasible Space Approximation}
Recall that $\mathcal{F}$ is a compact subspace of the space of continuous functions on $[0,T]$. Let $\mathcal{F}_n$ denote an $n$-dimensional ($n < \infty$) closed subspace of $\mathcal{F}$ such that elements in $\mathcal{F}$ can be approached by a sequence of elements in $\mathcal{F}_n$, that is, for every $F \in \mathcal{F},$ there exists $\{F_n, n \geq 1\}, F_n \in \mathcal{F}_n$ such that $\|F_n -F\| \to 0$. An example of $\mathcal F_n$ is the span of the first $n$ Legendre polynomials~\citep[pp. 176]{1989kre} on the interval $[0,T]$. More generally, $\mathcal{F}_n$ can be chosen as the span of the first $n$ elements of any Schauder basis of $\mathcal{F}$. (Recall that a sequence $\{P_j, j \geq 1\}$ of vectors in a normed space $\mathcal{F}$ is called a \emph{Schauder basis} of $\mathcal{F}$ if for every $F\in \mathcal F$ there is a unique sequence $\{a_j, j \geq 1\}$ of scalars such that $\|F - \sum_{j=1}^{n} a_j P_j \| \to 0$ as $n \to \infty$.) Consequently, we assume that

\begin{assumption}\label{ass:distsubspace}
    The closed finite-dimensional function subspace $\mathcal F_n \subset \mathcal{F}$ is such that 
    \begin{equation} \label{ass:findimconv} \psi(n) := \sup_{F \in \mathcal{F}} \left\| F - \Pi_{\mathcal{F}_n}(F) \right\| = O(g(n)),
    \end{equation} 
where $g(n) \to 0$ as $n \to \infty.$ For example, if $\mathcal{F}_n$ is spanned by Legendre polynomials, we expect
$g(n) = \mathcal{O}\bigl(n^{-\alpha}\bigr)$ for some smoothness parameter~$\alpha$. 
\end{assumption}

\subsection{Stochastic Process Approximation}
We will assume that there exists a method for generating a sequence of stochastic processes $\{Z_h : h > 0\}$ that satisfies the following `weak convergence' condition.

\begin{assumption}~\label{ass:weak-order}
	The method used to generate paths $Z_{h} {\sim} \pi^{x}_{h}$ exhibits weak convergence order $\beta > 0$ to the process $Z\sim \pi^x$ as $h \to 0$, implying that there exists $\ell_1<\infty$ such that \begin{equation}\label{discbd}\sup_{F \in \mathcal{F}} \left|\mathbb{E}\left[\tilde{J}\circ\Gamma(F+Z_h)\right] - J(F)\right| \leq \ell_1 h^\beta.\end{equation} 
%	where $Y^F_h = F + Z_h$.
\end{assumption}

 In the case where $Z$ is a Brownian motion, $Z_h$ can be generated, for example, in the following ways.

\begin{itemize}
    \item 
    \it{Euler Schemes.} Both the Euler-Maruyama and Euler-Milstein schemes for solving stochastic differential equations~\cite[Ch. X]{Asmussen_2007} exhibit a weak convergence order of $h=1$, over the partition points $0 = t_0 < t_1 < t_2 < \cdots < t_{n-1} = T,$ where $h = h(n):= \max\{t_1 - t_0, t_2-t_1, \ldots, t_n - t_{n-1}\}$.

    \item
    \it{Wong-Zakai Approximation.} Suppose we use a piecewise smooth, path-by-path approximation of the Brownian motion, then the Wong-Zakai theorem~\citep{Twardowska1996} shows that the solution of the corresponding (controlled) ordinary differential equation approximates the Stratonovich solution of a given SDE, with a weak convergence order of $\beta=1$ under sufficient regularity assumptions on the coefficients. In our setting, the Ito and Stratonovich solutions coincide. A piecewise smooth approximation of Brownian motion can be computed, for instance, by using $N = \lceil 1/h\rceil$ terms in the Haar wavelet expansion in L\'evy's construction or more generally using smoother bases such as Daubechies wavelets~\citep{Grebenkov2016} or using a polynomial basis~\citep{foster2020optimal}.
\end{itemize}

With the above notation in place, the SAA problem (MC-$n$-OPT) approximating (OPT) is:
\begin{align}\label{sampathpb}
%\mbox{min. }J(F) &:= \mathbb{E}[\tilde{J}(X^F)] = a_1 \int_0^T \mathbb{E}g(X^F_t) dt + a_2 \mathbb{E}G(X^F_T)), & \mbox{s.t. } F \in \mathcal{F} \tag{OPT}\\ 
%\mbox{min. }J_{n}(F) &:= J(F), & \mbox{s.t. } F \in \mathcal{F}_n \tag{$n$-OPT} \\
%\mbox{min. } & \left\{ J_{N,h}(F) := \frac{1}{N} \sum_{j=1}^N \tilde{J}(X_{h,j}^F)\right\}, \quad X^F_{h,j} \overset{\scriptsize \mbox{iid}}{\sim} \pi_{F,h}^{\Gamma} \nonumber \\ \mbox{s.t. } & F \in \mathcal{F}_n \tag{MC-$n$-OPT}
& \mbox{min.} \left\{ J_{N,h}(F) := \frac{1}{N} \sum_{j=1}^N \tilde{J}\circ \Gamma(Z_{h,j}+F)\right\}, \nonumber \\
& Z_{h,j} \overset{\scriptsize \mbox{iid}}{\sim} \pi^x_{h} \nonumber \\ & \mbox{s.t. } F \in \mathcal{F}_n. \tag{MC-$n$-OPT}
\end{align} 
%where the measure 
% %$\pi_{F,h}^{\Gamma}$ 
% $\pi_{0,h}$ approximates the measure $\pi^x$.
For brevity, we will write $Z_{h,j}$ as $Z_h$ in the remainder of this section. To ensure theoretical guarantees on the solution quality of~\eqref{sampathpb} as an estimator to the solution to (OPT), we assume that \begin{assumption}\label{ass:Jtildeconv} The random functional %$\tilde{J}:\mathcal{F} \to \mathbb{R}$ 
$F \mapsto \tilde J\circ \Gamma (Z+F)$ is convex in $F$.
\end{assumption} 
We define the following optimal values and optimal solution (sets) corresponding to (OPT) and (MC-$n$-OPT), the existence of which will become evident. \begin{align}\label{morenot} J^* &:= \inf_{F \in \mathcal{F}} \{J(F)\}; \quad \mathcal{F}^* {:=} \underset{F \in \mathcal{F}}{\arg\inf} \{J(F)\} \\% J_{n}^* &:= \inf_{F \in \mathcal{F}_n} \{J(F)\}; \quad \hat{\mathcal{F}}_{n}^* {:=} \underset{F \in \mathcal{F}_n}{\arg\inf}\{J(F)\} \\ 
J_{N,n}^* &:= \inf_{F \in \mathcal{F}_n} \{J_{N,h}(F)\}; \quad \mathcal{F}_{N,n}^* {:=} \underset{F \in \mathcal{F}_n}{\arg\inf} \{J_{N,h}(F)\}. \nonumber \end{align}
%\hl{should we use $J_{N,n,h}^* $?}
It is important that the optimization in (MC-$n$-OPT) be performed over a finite-dimensional subspace $\mathcal{F}_n$ of $\mathcal{F}$ so as to allow computation with methods such as gradient descent~\citep{2004nes}. Also, in~\eqref{morenot}, notice that we have suppressed the dependence of $J^*_{N,n}$ and $\mathcal{F}^*_{N,n}$ on $h$ used to generate Monte Carlo samples from the measure $\pi^x_{h}$. 

 We call any solution $F_{N,n}^* \in \mathcal{F}^*_{N,n}$ to~\eqref{sampathpb} an SAA estimator of the solution to (OPT). An SAA estimator cannot be obtained in ``closed analytical form" in general. However, given that~\eqref{sampathpb} is a deterministic convex optimization problem over a closed finite-dimensional subspace, one of various existing iterative techniques, e.g., mirror descent~\citep{2015bub}, can be used to generate a sequence $\{F_{N,n,k}^*, k \geq 1\} \subset \mathcal{F}_{N,n}$ that converges to a point in $\mathcal{F}_{N,n}^*$, that is, $F_{N,n,k}^* \to \mathcal{F}_{N,n}^*$ as $k \to \infty$ for fixed $N,n,h$.  Before we present the main result that characterizes the accuracy of $F^*_{N,n,k}$, we state a lemma that will be invoked. 

\begin{lemma}\label{lem:supvarbd} Let Assumption~\ref{ass:LipschitzinF} hold, and suppose there exists $F_0 \in \mathcal{F}$ such that \begin{equation}\label{finvar} \sigma_0^2(h) := \mbox{Var}(\tilde{J}\circ \Gamma(Z^{}_{h} + F_0)) < \infty; \quad Z_{h} \overset{\mbox{\scriptsize iid}}{\sim} \pi^{x}_{h}.\end{equation} Then, 
\begin{equation} \label{varbd} 
\begin{aligned}
&\sup_{F \in \mathcal{F}} \mbox{Var}(\tilde{J}\circ \Gamma(Z^{}_{h} + F_0))\\
&\qquad\leq \left(\sigma_0(h) + \mbox{diam}(\mathcal{F})\sqrt{\mathbb{E}[K_{Z} L_\Gamma^2]} \right)^2.   
\end{aligned}
\end{equation} 
\end{lemma}
\begin{proof}{Proof}  We can write \begin{align}\label{initsplit}\tilde{J}\circ \Gamma(Z^{}_{h} + F) &= \tilde{J}\circ \Gamma(Z^{}_{h} + F_0) + \tilde{J}\circ \Gamma(Z^{}_{h} + F)\\\nonumber &\qquad\qquad\qquad\quad - \tilde{J}\circ\Gamma(Z^{}_{h} + F_0),\end{align} and due to Assumption~\ref{ass:LipschitzinF} and Assumption~\ref{ass:GammaLipschitz},  \begin{equation}\label{tijtilde}\left|\tilde{J}\circ\Gamma(Z^{}_{h} + F) - \tilde{J}\circ\Gamma(Z^{}_{h} + F_0)\right| \leq K_{Z} L_\Gamma\, \mbox{diam}(\mathcal{F})\end{equation}
%\zzcomment{should the lipschitz constant be $K_{Z} L_\Gamma$ or $K_{Z_h} L_\Gamma$}
where $\mathbb{E}[K_{Z} L_\Gamma^2]<\infty.$ From~\eqref{tijtilde} we see that \begin{align}\label{vardiff}\mbox{Var}&(\tilde{J}\circ\Gamma(Z^{}_{h} + F) - \tilde{J}\circ\Gamma(Z^{}_{h} + F_0)) \\\nonumber &\qquad\qquad\qquad\qquad \leq \mathbb{E}\left[ K_{Z} L_\Gamma^2 \right]  \, \mbox{diam}^2(\mathcal{F}).\end{align} Use~\eqref{initsplit} and~\eqref{vardiff} along with~\eqref{finvar} to conclude that the assertion of the lemma holds. \end{proof}

We now present the main rate result governing the solution estimator $F^*_{N,n,k}$ of~\eqref{sampathpb}. 

\begin{theorem}~\label{thm:saarate} Suppose Assumptions~\ref{ass:LipschitzinF},~\ref{ass:weak-order},~\ref{ass:distsubspace} and~\ref{ass:Jtildeconv} hold. Furthermore, suppose mirror descent~\citep[pp. 80]{2015bub} is executed for $k$ steps on \eqref{sampathpb}: \begin{align} \label{seqest} 
\nabla \psi(G_{N,n,j+1}) &= \nabla \psi(F_{N,n,j}) - \eta \tilde{S}_{J_{N,h}}(F_{N,n,j});\\
&\quad j = 0,1,\ldots,k-1 \nonumber\\ 
F_{N,n,j+1} &= \sup_{w \in \mathcal{F}_n \cap \mathcal{D}} D_{\psi}(w,G_{N,n,j+1}); \nonumber \\
F^*_{N,n,k} &:= \frac{1}{k} \sum_{j=1}^k F_{N,n,j},\end{align}  where $\psi: \mathcal{D} \subset \mathcal{F} \to \mathbb{R}$ is a chosen $\rho$-strongly convex, mirror-map (see Definition~\ref{def:mirror}) with $\mathcal{F}_n \cap \mathcal{D} \neq \emptyset$, the Bregman divergence $D_{\psi}(w,w') := \psi(w) - \left(\psi(w') + \langle \nabla \psi(w'), w-w'\rangle \right),~\forall w,w' \in \mathcal{D},$ and the step size $\eta = \eta_0 \frac{R}{\bar{K}} \sqrt{\frac{2\rho}{k}},~\eta_0 \in (0,1)$ with $R^2 := \sup_{w \in \mathcal{F}_n \cap \mathcal{D}} \psi(w) - \psi(F_{N,n,0})$, and $\bar{K} := N^{-1} \sum_{j=1}^N K_{Z_j}$ is the i.i.d. sample mean of  Lipschitz constants $K_{Z_j}, j=1,2,\ldots,N$ appearing in Assumption~\ref{ass:LipschitzinF} satisfying $\sup_{F \in \mathcal{F}} \| \tilde{S}_{J_{N,h}}(F)\|_{*} \leq \bar{K}; 
\tilde{S}_{J_{N,h}}(F) \in \partial J_{N,h}(F);\quad \mathbb{E}[K^2_{Z_j}] <\infty,$ 
% \begin{equation*} 
% \begin{aligned}
% &\sup_{F \in \mathcal{F}} \| \tilde{S}_{J_{N,h}}(F)\|_{*} \leq \bar{K};\\
% &\tilde{S}_{J_{N,h}}(F) \in \partial J_{N,h}(F);\quad \mathbb{E}[K^2_{Z_j}] <\infty, 
% \end{aligned} 
% \end{equation*}
where $\tilde{S}_{J_{N,h}(F)}$ is a subgradient and $\partial J_{N,h}(F)$ the subdifferential of the convex functional $J_{N,h}$ at the point $F$.  Then, for all $k \geq 1$, \begin{equation}
\begin{aligned}
0 &\leq \mathbb{E}\left[ J(F^*_{N,n,k}) - J(F^*)\right]\\
&\quad\leq \frac{c_1}{\sqrt{k}} + \frac{c_2}{\sqrt{N}} + c_3h^{\beta} + c_4 g(n),    
\end{aligned}
\end{equation} 
where
%\begin{align*}
\(
c_1 = \sqrt{\frac{2}{\rho}} \, \left(\mathbb{E}[R^2] \left( \frac{1}{k}\mbox{Var}(K_Z) + \mathbb{E}[K_Z^2]\right)\right)^{1/2},~c_2 = \frac{3 }{\sqrt{N}}\left(\mbox{diam}(\mathcal{F})\sqrt{\mathbb{E}[K_{Z} L_\Gamma^2] } + \sigma_0(h) \right),~ c_3 = \ell_1,~\emph{\mbox{ and }}~c_4 = \mathbb{E}\left[K_Z\right].
\)
%\end{align*}
\end{theorem} 

\begin{proof}{Proof} %Let's first recall some notation. \begin{align} F^* & \in \arg\min\{ J(F): F \in \mathcal{F}\}; \nonumber \\
% F_n^* & \in \arg\min\{ J(F): F \in \mathcal{F}_n\}; \nonumber \\ 
% F^*_{N,n} & \in \arg\min\{ J_{N,h}(F): F \in \mathcal{F}_n\}, \nonumber \end{align} where $\mathcal{F}_n$ is a finite-dimensional subspace of $\mathcal{F}$. 

%Also, let $\tilde{F}^*_n \in \arg\min\{\|F - F^*\|: F \in \mathcal{F}_n\}$ be the projection of $F^*$ on $\mathcal{F}_n$, that is, $$\tilde{F}_n^* : = \arg\min\{\|F-F^*\|: F \in \mathcal{F}_n\}.$$ (Since $\mathcal{F}_n \subset \mathcal{F}$ is a closed subspace of the normed space $\mathcal{F}$, the selection $\tilde{F}^*_n$ exists.) 
%\hl{RP:clarify} Also, denote $u_n^* := F^* - \tilde{F}_n^*$. Using notation from Section~\ref{sec:saarate}, we 
Observe that  
\begin{equation}
\begin{aligned} \label{fintimebd} 
0 &\leq J(F^*_{N,n,k}) - J(F^*) %& = J(F^*_{N,n,k}) - J_{N,h}(F^*_{N,n,k}) + J_{N,h}(F^*_{N,n,k}) - J(F^*)  \nonumber
\\ & = J(F^*_{N,n,k}) - J_{N,h}(F^*_{N,n,k}) + J_{N,h}(F^*_{N,n,k})\\
&\qquad- J_{N,h}(F^*_{N,n}) + J_{N,h}(F^*_{N,n}) - J(F_{n}^*)\\
&\qquad+ J(F_{n}^*) - J(F^*) \\ 
& \leq
J_{N,h}(F^*_{N,n,k}) - J_{N,h}(F^*_{N,n}) \\
&\qquad +\sum_{F \in \{F^*_{N,n,k}, F^*_{N,n}, F_{n}^*\}} |J_{N,h}(F) - J(F)|\\
&\qquad + J(F_{n}^*) - J(F^*) \\ 
& \leq
\underbrace{J_{N,h}(F^*_{N,n,k}) - J_{N,h}(F^*_{N,n})}_{\mbox{\scriptsize opt. error}}\\
&\qquad+ \sum_{F \in \{F^*_{N,n,k}, F^*_{N,n}, F_{n}^*\}} \underbrace{|J_{N,h}(F) - \mathbb{E}\left[J_{N,h}(F)\right]|}_{\mbox{\scriptsize sampling error}}   \\
&\qquad+ \sum_{F \in \{F^*_{N,n,k}, F^*_{N,n}, F_{n}^*\}} \underbrace{|\mathbb{E}\left[J_{N,h}(F)\right] - J(F)|}_{\mbox{\scriptsize approx. error}}\\
&\qquad-\underbrace{S_J(F^*_{n}) \, \|F_{n}^* - F^*\|}_{\mbox{\scriptsize proj. error}},
\end{aligned}    
\end{equation}

where the penultimate inequality in~\eqref{fintimebd} follows from rearrangement of terms %and noticing that $$J_{N,h}(F^*_{N,n}) - J(F_{n}^*) \leq \max(|J_{N,h}(F^*_{N,n}) - J(F^*_{N,n})|,|J_{N,h}(F^*_{n}) - J(F^*_{n})|),$$ 
and the last inequality follows upon using the sub-gradient inequality \eqref{eq:subgrad} for the convex functional $J(\cdot)$. 
%, that is, \begin{align}\label{lastineq}J(F_{n}^*) - J(F^*) &\leq - S_J(F^*_{n})\|F_{n}^*-F^*\|.\end{align} 
Now we quantify (in expectation) each of the error terms appearing on the right-hand side of~(\ref{fintimebd}). Applying mirror descent's complexity bound~\citep[pp. 80]{2015bub} on the $\bar{K}$-smooth function $J_{N,h}(\cdot)$ and taking expectation, we get %\begin{equation}\label{gradbound} 0 \leq J_{N,h}(F^*_{N,n,k}) - J_{N,h}(F^*_{N,n}) \leq \frac{1}{\sqrt{k}} \, R\bar{K} \sqrt{\frac{2}{\rho}}.\end{equation} Taking expectations in~\eqref{gradbound}, we obtain:
%get the bound on optimization error, that is, the first of the three terms on the right-hand side of~(\ref{fintimebd}):
\begin{equation}\label{expgradbound} 
\begin{aligned}
0 &\leq \mathbb{E}\left[ J_{N,h}(F^*_{N,n,k}) - J_{N,h}(F^*_{N,n})\right]\\&\leq \frac{1}{\sqrt{k}} \, \sqrt{\frac{2}{\rho}} \, \left(\mathbb{E}[R^2] \left( \frac{1}{k}\mbox{Var}(K_{z}) + \mathbb{E}[K_{z}^2]\right)\right)^{1/2}.    
\end{aligned}
\end{equation} Next, 
%since Lemma~\ref{lem:supvarbd} implies that \(\label{varbd2} \sup_{F \in \mathcal{F}} \mbox{Var}(\tilde{J}(\Gamma(Y+F))) \leq \left(\mbox{diam}(\mathcal{F})\sqrt{\mathbb{E}[K_Z^2] + (\mathbb{E}[K_Z])^2} + \sigma_0(h) \right)^2,\)
%and recalling that for $F \in \mathcal{F}$, $J_{N,h}(F) = \frac{1}{N} \sum_{j=1}^N \tilde{J}(\Gamma(Y_j+F)); \quad Z_j \overset{\mbox{\scriptsize iid}}{\sim} \pi_{0,h},$ 
using Lemma~\ref{lem:supvarbd} we get the bound on approximation error in~(\ref{fintimebd}):  
\begin{equation}\label{mcerror} 
\begin{aligned}
&\mathbb{E}\left[\sum_{F \in \{F^*_{N,n,k}, F^*_{N,n}, F_{n}^*\}} \left | J_{N,h}(F) - \mathbb{E}\left[J_{N,h}(F)\right]\right | \right]\\
&\qquad\leq \frac{3 }{\sqrt{N}}\left(\mbox{diam}(\mathcal{F})\sqrt{\mathbb{E}[K_{Z} L_\Gamma^2]} + \sigma_0(h) \right).    
\end{aligned}
\end{equation}
Due to the assumption in~\eqref{discbd}, we have \begin{equation}\label{discbd2}\sup_{F \in \mathcal{F}} |\mathbb{E}\left[J_{N,h}(F)\right] - J(F)| \leq \ell_1 h^\beta.\end{equation} And since $J$ is convex, we see that
\begin{equation}\label{subspacebd} 
\begin{aligned}
J(F^*_n) - J(F^*) & \leq  \|S_J(F^*_n)\|_{*} \, \|F^*_n - F^*\| \\&\leq \sup_{F \in \mathcal{F}} \, \|S_J(F)\|_{*} \, \|F^*_n - F^*\|\\
&\leq \mathbb{E}\left[K_{z}\right] \, g(n),
\end{aligned}     
\end{equation}

where the last inequality in~\eqref{subspacebd} is from Assumption~\ref{ass:LipschitzinF}.  Now use~\eqref{expgradbound},~\eqref{mcerror},~\eqref{discbd2}, and~\eqref{subspacebd} to conclude.
 \end{proof}
Next, we illustrate this result by considering the computation of the pathwise directional derivative in Proposition~\ref{prop:Sko-pathwise}.

\begin{corollary}\label{cor:saarate}
In the setting of Theorem~\ref{thm:saarate}, suppose in addition that  the regulator is the Skorokhod map. If $\tilde{S}_{J_{N,h}(F)}$ is chosen to be the sample average gradient of the convex functional $J_{N,h}$ at the point $F$, i.e., for basis functions $\{P_j, 1\leq j\leq n\}$ spanning the subspace of choice, $\mathcal{F}_n$, let
\begin{align*}
&\tilde{S}_{J_{N,h}(F)}\\
&=\nabla \hat{\mathbb E}[\tilde{J}\circ\Gamma(Z_h + F)] \\
&=\Big[{D}_{P_1} \hat{\mathbb E}[\tilde{J}\circ\Gamma(Y)],\cdots,{D}_{P_n} \hat{\mathbb E}[\tilde{J}\circ\Gamma(Y)]\Big]^T,
\end{align*}
where
\begin{align*}
    {D}_u \hat{\mathbb E}[\tilde{J}\circ\Gamma(Y)] = \frac{1}{N}\sum_{i=1}^N D_u\tilde{J}\circ\Gamma(Y_i),
\end{align*}
\end{corollary}
then Then, for all $k \geq 1$,
\begin{equation}
\begin{aligned}
0 &\leq \mathbb{E}\left[ J(F^*_{N,n,k}) - J(F^*)\right]\\
&\quad\leq \frac{c_1}{k} + \frac{c_2}{\sqrt{N}} + c_3h^{\beta} + c_4 g(n),    
\end{aligned}
\end{equation}
with $c_1 = \frac{\mathbb{E}[R^2]}{\rho}$ and the same set of constants $c_2$ to $c_4$ as in Theorem~\ref{thm:saarate}.
\begin{proof}{Proof}
The result follows from the proof of Theorem~\ref{thm:saarate} with additional application of the $O(1/k)$ smooth convex mirror descent bound\citep[Theorem 4.4]{2015bub}.
 \end{proof}

\section{Asymptotically Optimal Allocation}~\label{sec:allocate}
%\subsection{Optimal parameter choices}
In the previous section, we characterized the optimality gap between the true optimal solution to (OPT) and the solution to the computational objective~\eqref{sampathpb}:
for all $k \geq 1$, 
\begin{equation}\label{saafinbd} 
\begin{aligned}
0 &\leq \mathbb{E}\left[ J(F^*_{N,n,k}) - J(F^*)\right]\\
&\leq \frac{c_1}{\sqrt{k}} + \frac{c_2}{\sqrt{N}} + c_3h^{\beta} + c_4 g(n),    
\end{aligned}
\end{equation}
Given a finite computational budget of size $B$ (this could represent computation time or number of resources that can be translated into time units), there are tradeoff's to be made in allocating this budget across the various computations that need to be performed. We first provide a break down of the computational burden of each of the computations that need to be performed. 

\subsection{Computational Burden}
Observe that the total computational burden of our method consists of three parts: the sampling step performed at the beginning, the gradient computation step at each iteration, and the mirror step at each iteration. 

\begin{itemize}
\item {\bf Stochastic simulation.} 
In general, each sample path $Z_h$ requires the generation of $O\left( \frac{1}{h} \right)$ random variates at approximation level $h > 0$. For example, if an Euler scheme is used to generate $Z_h$ approximating standard Brownian motion $Z$, then $\lceil\frac{T}{h}\rceil$ random variates per sample path are required; if a Levy-type construction with wavelets is used, then $\lceil\frac{T}{h}\rceil$ standard Gaussian random variates are necessary per sample path. Therefore, the total sampling complexity for sampling $N$ paths is $O\left(\frac{N}{h}\right)$. 

\item {\bf Sub-gradient computation.} Observe that the computation of the sub-gradient $\tilde S_{J_{N,h}}$ requires the pathwise derivative for each simulation path. To compute the latter, we need to first regenerate the sample paths $\{Z_{h,i} + F_{N,n,j},\,i=1,\ldots,N\}$ with the updated control $F_{N,n,j}$, which requires $O\left(\frac{N}{h}\right)$ basic operations. We next  compute the pathwise derivative for each of the $N$ paths, which requires $O\left(\frac{N}{h} \right)$ operations, in general. For instance, considering the directional derivative of the Skorokhod regulated path in~Proposition~\ref{prop:Sko-pathwise}, we need to track the stopping time process $\Phi(t)(\cdot)$ for each of the $N$ paths to compute the pathwise gradient, which requires $O\left(\frac{N}{h}\right)$ operations in total. Finally, we repeatedly evaluate the directional derivative along $n$ basis directions to estimate the gradient. Therefore, the  complexity of gradient approximation is of order $O\left(\frac{nN}{h}\right)$. This is further repeated for each iteration of the mirror descent procedure, resulting in a total complexity of $O\left(\frac{knN}{h}\right)$

\item {\bf Mirror descent.} The mirror step has the complexity of a standard Euclidean projection step~\citep{2015bub}, which is of order $O(n)$ at each iteration, and the total complexity is $O(kn)$ over $k$ iterations.
\end{itemize}
Thus, the total computational complexity is $O\left(\frac{N}{h} + \frac{knN}{h} + kn\right)=O\left(\frac{knN}{h}\right)$. 

\subsection{Budget Allocation Problem}
Recall that $B$ denotes the total computational budget. The following optimization problem models the computational tradeoffs and seeks an (asymptotically) optimal budget allocation across the computational steps listed above:%\zzcomment{We did not analyze any asymptotic behavior of the budget allocation in this subsection though.}\hl{HH: budget constraint is asymptotic order only, can only hold for large B.}       
\begin{align*}
&\min_{n, N, k\in \mathbb{Z}^+, h>0} \frac{c_1}{\sqrt{k}} + \frac{c_2}{\sqrt{N}} + c_3h^{\beta} + {c_4 } g(n)\\
&\text{s.t.}\qquad\frac{knN}{h}=B.
\end{align*}
Eliminating the constraint by expressing $h$ as a function of other variables, yields
\begin{align*}
\min_{n, N, k\in \mathbb{Z}^+} \frac{c_1}{\sqrt{k}} + \frac{c_2}{\sqrt{N}} + c_3\left(\frac{nkN}{B}\right)^{\beta} + {c_4} g(n).
\end{align*}
The optimal allocation clearly depends on $g(n)$, which is problem specific. To further the analysis, assume that the function space approximation rate satisfies $g(n)={n^{-\alpha}}$ for some $\alpha>0$. For example, the piecewise linear approximation is of convergence order $o(n)$, equivalent to $\alpha=1$. While the objective is non-convex for arbitrary choices of $\beta >0$, it is necessarily convex for any $\beta\geq 1$ and $\alpha >0$. That is, we are asking for a simulation scheme with a weak convergence order greater than or equal to $1$. As noted before, in the case where $Z$ is Brownian motion, for example, the Euler–Maruyama and Euler-Milstein schemes satisfy $\beta = 1$.  

Now, straightforward calculations, show that the optimal solution is
%\begin{align*}
\(
\log_{B} k^* = \Xi_1 + {\frac{2\alpha\beta}{\alpha+\beta+4\alpha\beta}},~
\log_B N^* = \Xi_2 + {\frac{2\alpha\beta}{\alpha+\beta+4\alpha\beta}},
\log_B n^* = \Xi_3 + {\frac{\beta}{\alpha+\beta+4\alpha\beta}}~\text{and}~   
\log_B h^* = (\Xi_1+\Xi_2+\Xi_3) {-\frac{\alpha}{\alpha+\beta+4\alpha\beta}},
\)
%\end{align*}
where $\Xi_1$, $\Xi_2$ and $\Xi_3$ are constants that depend on $c_1$, $c_2$, $c_3$, $c_4$, $\alpha$ and $\beta$.
The exact form of the constants is complex, however we can illustrate a specific case by fixing $\beta=1$ and $\alpha=1$, and solve the corresponding optimization problem to obtain
%\begin{align*}
\(
k^* = (2^{-\frac{2}{3}}c_1^{\frac{4}{3}}  c_2^{-\frac{2}{3}} c_3^{-\frac{1}{3}}  c_4^{-\frac{1}{3}} )B^{\frac{1}{3}},~
N^* = (2^{-\frac{2}{3}}c_1^{-\frac{2}{3}}  c_2^{\frac{4}{3}} c_3^{-\frac{1}{3}}  c_4^{-\frac{1}{3}} )B^{\frac{1}{3}},~\text{and}~
n^* = (2^{\frac{2}{3}}c_1^{-\frac{1}{3}}  c_2^{-\frac{1}{3}} c_3^{-\frac{1}{6}}  c_4^{\frac{5}{6}} )B^{\frac{1}{6}}.
\)
%\end{align*}
Note that the row echelon form of the Hessian matrix at the critical point is an identity matrix, implying that the second-order optimality condition is satisfied. Thus, this critical point is a global minimum. Substituting the solutions, we get
%\begin{align*}
\(
h^* = (2^{-\frac{2}{3}}c_1^{\frac{1}{3}}  c_2^{\frac{1}{3}} c_3^{-\frac{5}{6}}  c_4^{\frac{1}{6}} )B^{-\frac{1}{6}}.
\)
%\end{align*}

\subsection{Discussion}
The optimal parameter choices shed light on how to design the simulation when implementing our SAA framework. 
{It is instructive to consider the exponent for $k^*~\text{and}~N^*$, 
\(
\gamma := \frac{2\alpha\beta}{\alpha+\beta+4\alpha\beta}.
\)
Dividing by $4\alpha\,\beta$ in the numerator and denominator, we get
\(
\gamma = \frac{1}{2}{\left(\frac{1}{4}\left(\frac{1}{\beta} + \frac{1}{\alpha} \right) + 1\right)^{-1}}.
\)
Recall that the canonical budget allocation rates for mirror descent and SAA in the finite-dimensional setting is $\gamma = 1/2$.%\zzcomment{where does $\gamma = 1/2$ come from, in our context, if $\alpha=1$ and $\beta=1$, $\gamma=\frac{1}{3}$, does it mean that we pay more price than a "canonical mirror descent and SAA"?} \hl{HH: Yes. Remember in the usual case, there are no such approximations.} 
~The denominator in $\gamma$ is the `price' we pay to compute the approximations in the infinite-dimensional setting (with the resulting expressions for $n^*,k^*$), limiting our ability to perform estimation and optimization with a given budget $B$.}

Further illustration is available by considering asymptotic regimes for the function approximation rate. Suppose the feasible domain is a “nice” function space for which subspace convergence to the original space is fast (i.e., $\alpha\rightarrow\infty$). In this limit, we have
%\begin{align*}
\(
\log_B k^* = \Xi_1 + {\frac{2\beta}{1+4\beta}},~
\log_B N^* = \Xi_2 + {\frac{2\beta}{1+4\beta}},~
\log_B n^* = \Xi_3,~\text{and}~
\log_B h^* = \Xi_1+\Xi_2+\Xi_3 {-\frac{1}{1+4\beta}}.
\)
%\end{align*}
Observe that $n^*$ is constant; thus, more computational effort should be allocated to sampling, gradient estimation, and optimization. In contrast, when $\alpha\to 0$,  we obtain
%\begin{align*}
\(
\log_B k^* = \Xi_1,~
\log_B N^* = \Xi_2,~
\log_B n^* = \Xi_3 + {1},~\text{and}~   
\log_B h^* = \Xi_1+\Xi_2+\Xi_3.
\)
%\end{align*}
This result indicates that (in the large budget limit) we should allocate almost the entire computational budget to the function space approximation itself, and only allocate a limited/fixed amount to sampling and gradient estimation, and optimization. This highlights a key takeaway from this analysis: the approximation rate parameter $\alpha$ reflects the smoothness of the function space, and the ``rougher'' the function space, the larger the share of the computational budget that needs to be allocated to the function approximation alone, thereby limiting optimization and estimation accuracy. 

\section{Conclusions and Future Directions}~\label{sec:conc}
This paper introduces a drift optimization framework for a broad class of functional stochastic optimization problems driven by regulated stochastic processes. Central to our approach is the sample average approximation (SAA) methodology, which we extended to infinite-dimensional function spaces through a careful integration of pathwise discretization, function-space approximation, and Monte Carlo sampling.

We derive the pathwise directional derivative for the Skorokhod regulator using Danskin’s theorem. This enabled an unbiased gradient estimator suitable for simulation-based optimization schemes such as mirror descent. We establish consistency and convergence rates for our proposed SAA method, providing insight into the optimal budget allocation that minimizes the aggregate error under a large, but finite computational budget. Furthermore, we discuss the asymptotic behavior of the optimal budget allocation, highlighting how the smoothness of the feasible function space informs the prioritization of computational effort.

By delaying discretization until necessary, our methodology respects the functional structure of the problem by leveraging the function form pathwise directional derivative, while remaining practical for implementation. This work paves the way for future research in stochastic optimization with regulated dynamics, particularly in settings that demand scalable, structure-aware simulation methods. Promising directions include extensions to feedback (closed-loop) control, richer classes of regulators and constraints arising in queuing systems, variational inference, and optimal transport.

% Acknowledgments here
% Acknowledgments here
\section{Acknowledgements}
The authors gratefully acknowledge the following grant support. ZZ is supported by the National Science Foundation (NSF) through grant DMS/22036385. HH is partially supported by NSF through grants DMS/22036385 and CMMI/22014426.

% References here (outcomment the appropriate case)
\bibliographystyle{unsrtnat}
%\bibliography{refs,refs2, raghu,demobib,stochastic_optimization} 
\bibliography{arXiv_submittion}

\newpage
\appendix
\section{Drift Optimization as a Control-theoretic Problem}~\label{sec:doss-sussmann}
 Let us recall the Doss-Sussmann transformation (\cite[Ch. V, Th. 25]{Protter},~\cite[Sec. 3.5.1]{Han_2017}]), as applied to a scalar diffusion.

    \begin{theorem}[Doss-Sussmann]
        Suppose the SDE
        \begin{align}\tag{$\star$}
        X_t = x + \int_0^t f(s,X_s,\theta_s) ds + \sigma B_t
        \end{align} 
        has a strong solution, and let $O_t$ denote the stationary, zero mean Ornstein-Uhlenbeck diffusion process satisfying the SDE $dO_t = - O_t dt + \sigma dB_t$. Then, the solution of the SDE $X_t$ is equivalent to the randomized ordinary differential equation (ODE),
        \begin{align*}
            \frac{d Z_t}{dt} = f(t,Z_t+O_t,\theta_t) + O_t,~Z_0 = x.
        \end{align*}
    \end{theorem}

    Since we have assumed that the diffusion process has a strong solution, it follows that $X$ has a pathwise unique solution. Consequently, every sample path of $(\star)$ is represented by a solution  $Z$ of the randomized ODE, for each $\omega \in \Omega$. It follows that OPT can be equivalently expressed in a control problem formulation as
    \begin{align*}
        &\text{minimize}~\mathbb E_{\mathbb P} \left[ \tilde J\circ\Gamma(Z) \right]\\
        &\text{subject to:}\\
        &\frac{d Z_t(\omega)}{dt} = f(t,Z_t(\omega)+O_t(\omega),\theta_t) + O_t(\omega),\\
        &Z_0 = x,\\
        &dO_t(\omega) = - O_t(\omega) dt + \sigma dB_t(\omega).
    \end{align*}
 %\end{example}

 \section{Proofs from Section~\ref{sec:dir-der}}
 \subsection{Proof of Proposition~\ref{prop:Sko-pathwise}}
 \begin{proof}{Proof}%[Proof of Lemma~\ref{prop:Sko-pathwise}]
We first observe that $$\Gamma(y)(t) = y(t) + \sup_{0\leq s\leq t}\{-y(s)\}_+ = \sup_{0\leq s\leq t}\max\{y(t)-y(s), y(t)\}.$$
In the view of Lemma~\ref{lem:5.0}, consider $\max\{y(t)-y(s), y(t)\}$ as $M(u,v)$ where $C$ is $U$ in Lemma~\ref{lem:5.0} and $[0, t]$ is $V$ in Lemma~\ref{lem:5.0}. $[0, t]$ is indeed compact.  $\sup_{0\leq s\leq t}\max\{y(t)-y(s), y(t)\}$ is $\Bar{M}(u)$. $\max\{y(t)+\delta u(t)-y(s) -\delta u(s), y(t)+\delta u(t)\}$ is continuous in s, implying upper semi-continuity. Fixing s, each term inside the maximum function is continuous in $\delta$, implying that the maximum function is also continuous in $\delta$ and hence upper semi-continuous. To verify the third assumption in Lemma~\ref{lem:5.0}, we compute the directional derivative of $\max\{y(t)+\delta u(t)-y(s)-\delta u(s), y(t)+\delta u(t)\}$ along direction $u$, which is given by

\begin{align*}
    &D_u \Bigg\{ y(t) + 
        \max\Big\{\delta u(t) - y(s) - \delta u(s),\delta u(t) \Big\} 
    \Bigg\} \\
    &= \lim_{\epsilon \to 0^+} \frac{1}{\epsilon} \Bigg( 
        y(t) + (\delta + \epsilon) \max\Big\{ (u(t) - u(s)), u(t) \Big\} \\
    &\qquad\quad - \max\Big\{ y(t)- y(s) + \delta (u(t)  - u(s)), \\
    &\qquad \qquad\quad y(t) + \delta u(t) \Big\} 
    \Bigg).
\end{align*}
Observe the two maximum functions is the directional derivative expression, consider four cases:
\begin{enumerate}
    \item 
    \begin{align*}
        &y(t)+(\delta+\epsilon)u(t)-y(s)-(\delta+\epsilon)u(s)\\&\geq y(t)+(\delta+\epsilon)u(t),\\
        &y(t)+\delta u(t)-y(s)-\delta u(s)\geq y(t)+\delta u(t).
    \end{align*}
    \item 
    \begin{align*}
        &y(t)+(\delta+\epsilon)u(t)-y(s)-(\delta+\epsilon)u(s)\\&\geq y(t)+(\delta+\epsilon)u(t),\\
        &y(t)+\delta u(t)-y(s)-\delta u(s)< y(t)+\delta u(t).
    \end{align*}\item 
    \begin{align*}
        &y(t)+(\delta+\epsilon)u(t)-y(s)-(\delta+\epsilon)u(s)\\&< y(t)+(\delta+\epsilon)u(t),\\
        &y(t)+\delta u(t)-y(s)-\delta u(s)\leq y(t)+\delta u(t).
    \end{align*}\item 
    \begin{align*}
        &y(t)+(\delta+\epsilon)u(t)-y(s)-(\delta+\epsilon)u(s)\\&< y(t)+(\delta+\epsilon)u(t),\\
        &y(t)+\delta u(t)-y(s)-\delta u(s)> y(t)+\delta u(t).
    \end{align*}
\end{enumerate}
In the first two cases, $D_u \max\{y(t)+\delta u(t)-y(s)-\delta u(s), y(t)+\delta u(t)\}=u(t)-u(s)$. In the other two cases, $D_u \max\{y(t)+\delta u(t)-y(s)-\delta u(s), y(t)+\delta u(t)\}=u(t)$. In all four cases, $D_u \max\{y(t)+\delta u(t)-y(s)-\delta u(s), y(t)+\delta u(t)\}$ is continuous in $s$.
Now we have verified all the assumptions in Lemma~\ref{lem:5.0}.
To get the final expression for $D_u\Gamma(y)(t)$, consider
\begin{align*}
    &D_u \max\{y(t)-y(s), y(t)\}\\
    &=\lim_{\epsilon\to 0^+}\frac{1}{\epsilon}\Bigg(\max\Big\{y(t)+\epsilon u(t)-y(s)-\epsilon u(s),\\& y(t)+\epsilon u(t)\Big\}- \max\Big\{y(t)-y(s), y(t)\Big\}\Bigg).
\end{align*}
Observe that when $y(s)<0$,  we can choose $\epsilon$ sufficiently small such that $\epsilon|u(s)|< -y(s)$, ensuring that the first argument inside the max function is dominant. In this case, $D_u \max\{y(t)-y(s), y(t)\}=u(t)-u(s)$. On the other hand, when $y(s)>0$, we can choose $\epsilon$ sufficiently small such that $\epsilon|u(s)|< y(s)$, ensuring that the second argument inside the max function is dominant. Thus, $D_u \max\{y(t)-y(s), y(t)\}=u(t)$. When $y(s)=0$, if $u(s)\geq 0$, we have $D_u \max\{y(t)-y(s), y(t)\}=u(t)$. If $u(s)<0$, then $D_u \max\{y(t)-y(s), y(t)\}=u(t)-u(s)$. 

By Lemma~\ref{lem:5.0},
\begin{align*}
    D_u\Gamma(y)(t) &= \sup_{s\in \Phi_t(y)} D_u \max\{y(t)-y(s), y(t)\}.
\end{align*}
Recall that 
\begin{align*}
    \Phi_t(y) :&= \{0\leq s \leq t : \Gamma(y)(t) = y(t)-y(s)\} \\
    &= \{0\leq s \leq t : \Psi(y)(t) = -y(s)\}\\
    &= \{0\leq s \leq t : \sup_{0\leq r \leq t} \{-y(r)\}_+ = -y(s)\}.
\end{align*}
The last expression implies that, for any $s\in \Phi_t(y)$, we must have $y(s)\leq 0$. Hence, any s where $y(s)>0$ will never be included in $\Phi_t(y)$.
For sample path $y(\cdot)$ defined on $[0, T]$ that starts from 0, if the path ever revisits 0, it will drop below 0 with probability 1. Thus, in such case $\Phi_t(y)$ contains only time points such that $y(s)<0$ and
\begin{align*}
    D_u\Gamma(y)(t) &= \sup_{s\in \Phi_t(y)} D_u \max\{y(t)-y(s), y(t)\}\\
    &= \sup_{s\in \Phi_t(y)} \{u(t) - u(s)\}\\
    &= u(t) + \sup_{s\in \Phi_t(y)} \{-u(s)\}.
\end{align*}
If the path remains strictly positive after the starting point, then $\Phi_t(y)=\{0\}$. In this case, when $u(0)\geq 0$, we obtain
\begin{align*}
    D_u\Gamma(y)(t) &= \sup_{s\in \Phi_t(y)} D_u \max\{y(t)-y(s), y(t)\}\\
    &= \sup_{s\in \Phi_t(y)} \{u(t)\}\\
    &= u(t).
\end{align*}
When $u(0)< 0$,
\begin{align*}
    D_u\Gamma(y)(t) &= \sup_{s\in \Phi_t(y)} D_u \max\{y(t)-y(s), y(t)\}\\
    &= u(t) + \sup_{s\in \Phi_t(y)} \{-u(s)\}.
\end{align*}
For sample path $y(t):0\leq t\leq T$ that starts below 0, $\Phi_t(y)=\{0\}$ is guaranteed to only contain $s$ such that $y(s)<0$, therefore,
\begin{align*}
    D_u\Gamma(y)(t) &= \sup_{s\in \Phi_t(y)} D_u \max\{y(t)-y(s), y(t)\}\\
    &= u(t) + \sup_{s\in \Phi_t(y)} \{-u(s)\}.
\end{align*}
For sample path $y(t):0\leq t\leq T$ that starts above 0, again it is a $0$ probability event that the path hits 0 but never becomes negative. For those paths that start above 0 and stay positive over $t: 0\leq t\leq T$, we have
\begin{align*}
    D_u\Gamma(y)(t) &= \sup_{s\in \Phi_t(y)} D_u \max\{y(t)-y(s), y(t)\}\\
    &= u(t).
\end{align*}
For paths that start above 0 and ever go below 0, we have
\begin{align*}
    D_u\Gamma(y)(t) &= \sup_{s\in \Phi_t(y)} D_u \max\{y(t)-y(s), y(t)\}\\
    &= u(t) + \sup_{s\in \Phi_t(y)} \{-u(s)\}.
\end{align*}
Combining all the six cases above, the result follows.

 \end{proof}
\subsection{Proof of Lemma~\ref{lem:5.2}}
\begin{proof}{Proof}%[Proof of Lemma~\ref{lem:5.2}]
Consider the case where $y(s):0< s\leq t$ attains negative values at least once. Let $T^*(t)$ denote the first time at which $\sup_{s\in\Phi_t(y)}$ is achieved. Then, we compute
\begin{align*}
    &\Psi(y+u)(t)-\Psi(y)(t)\\
    &=\sup_{0\leq s\leq t}\{-y(s)-u(s)\}_+-\sup_{0\leq s\leq t}\{-y(s)\}_+\\
    &= \left(-Z(T^*(t))-u(T^*(t))\right)_+-\sup_{0\leq s\leq t}\{-y(s)\}_+\\
    &\geq (-u(s)-y(s))-\sup_{0\leq s\leq t}\{-y(s)\}_+.
\end{align*}
For each $s \in \Phi_t(y)$, we have $(-u(s)-y(s))-\sup_{0\leq s\leq t}\{-y(s)\}_+ = -u(s)$, implying that
\begin{align}\label{eq:fr-lb}
	\Psi(y+u)(t)-\Psi(y)(t) \geq \sup_{s \in \Phi_t(y)}\{-u(s)\}.
\end{align}
On the other hand, 
\begin{align*}
    &\Psi(y+u)(t)-\Psi(y)(t)\\
    &=\sup_{0\leq s\leq t}\{-y(s)-u(s)\}_+-\sup_{0\leq s\leq t}\{-y(s)\}_+\\
    &\leq \sup_{0\leq s\leq t}\{(-y(s)-u(s))_+-(-y(s))_+\},
\end{align*}
where the inequality follows from the fact that for any $\mathbb R^d$-valued functions $f,g$, $$\sup_{0\leq s \leq t}\{f(s)\} \leq \sup_{0 \leq s \leq t}\{f(s)-g(s)\} + \sup_{0 \leq s \leq t}\{g(s)\}.$$
%Therefore, 
%$$\Psi(y+u)(t)-\Psi(y)(t)-\sup_{s\in\Phi_t(y)}\{-u(s)\} \leq \sup_{0\leq s\leq t}\{(-X(s)-u(s))_+-(-X(s))_+\}-\sup_{s\in\Phi_t(y)}\{-u(s)\}.$$
Now, since
%\begin{align*}
\(
   (-y(s)-u(s))_+
   =\max\{0,-y(s)-u(s)\}
   \leq\max\{0,-y(s)\}+\max\{0,-u(s)\},
\)
%\end{align*}
by the subadditivity property of the maximum function, it follows that 
\begin{align}
    \nonumber &\Psi(y+u)(t)-\Psi(y)(t)-\sup_{s\in\Phi_t(y)}\{-u(s)\}\\
    \nonumber &\leq \sup_{0\leq s\leq t}\{(-y(s)-u(s))_+-(-y(s))_+\}\\
    \nonumber &\qquad-\sup_{s\in\Phi_t(y)}\{-u(s)\}\\
    &\leq \sup_{0\leq s\leq t}\left\{-u(s)\right\}_+-\sup_{s\in\Phi_t(y)}\left\{-u(s)\right\}
    \label{eq:fr-ub}
    \leq 2||u||
\end{align}
Combining~\eqref{eq:fr-lb} and~\eqref{eq:fr-ub} we obtain
\begin{align*}
    -2||u||&\leq 0\\
    &\leq\Psi(y+u)(t)-\Psi(y)(t)\\
    &-\sup_{s\in\Phi_t(y)}\{-u(s)\} \leq 2||u||, 
\end{align*}
it follows that
\begin{align*}
    &|\Psi(y+u)(t)-\Psi(y)(t)-\sup_{s\in\Phi_t(y)}\{-u(s)\}|\\&
    =o( ||u||).
\end{align*}
Let us now consider when $y(0)=0$, $y(s)>0\,\forall \,s\in(0,t]$. Then, we compute
\begin{align*}
    &\Psi(y+u)(t)-\Psi(y)(t)\\
    &=\sup_{0\leq s\leq t}\{-y(s)-u(s)\}_+-\sup_{0\leq s\leq t}\{-y(s)\}_+\\
    &= \sup_{0\leq s\leq t}\{-y(s)-u(s)\}_+.
\end{align*}
This follows since the second supremum term is zero, given that $y(s)>0\,\forall \,s\in(0,t]$.
Since $y(s): 0<s\leq t$ is strictly positive,
we have $\sup_{0< s\leq t}\{-y(s)-u(s)\}_+ =0$ for sufficiently small $\|u\|$. On the other hand, if $u(0)<0$, $\{-y(0)-u(0)\}_+=\{-u(0)\}_+$ could attain a positive value. Thus, using the same argument in the last case, we have 
$|\Psi(y+u)(t)-\Psi(y)(t)-\sup_{s\in\Phi_t(y)}\{-u(s)\}|=o(||u||)$
for  $y(0)=0$, $y(s)>0\,\forall \,s\in(0,t]$ and $u(0)<0$.
In all the remaining cases, $\Psi(y+u)(t) = \Psi(y)(t) = 0$ for sufficiently small $\|u\|$ and we have
$|\Psi(y+u)(t)-\Psi(y)(t)|=o(||u||).$  
\end{proof}
\section{Proofs from Section~\ref{sec:equi}}\label{sec:equi-proofs}
\subsection{Proof of Lemma~\ref{lem:sub-gauss}}
\begin{proof}{Proof}%proof of ~\label{lem:sub-gauss}
Fix $F,G \in \mathcal{F}$ such that $F \neq G$. By H\"older's inequality we have
%\begin{align*}
\(
    |\mathcal{Y}_{F}(\mathbf Z) - \mathcal{Y}_G(\mathbf Z)| \leq \frac{1}{\sqrt{N}} \|g\|_{q} \|\mathcal{G}_F(\mathbf Z) - \mathcal{G}_G (\mathbf Z)\|_{p},
\)
%\end{align*}
where $\frac{1}{p} + \frac{1}{q} = 1$ and $q \geq 2$. Next, following Assumption~\ref{ass:LipschitzinF}, we have
\begin{align*}
    &\|\mathcal{G}_F - \mathcal{G}_G\|_{p}\\ 
    &\qquad= \left( \sum_{i=1}^N \left(\tilde{J}(Z_i+F) - \tilde{J}(Z_i+G)\right)^p\right)^{1/p}
    \\&\qquad\leq \left( \sum_{i=1}^N \left| K_{Z_i} \right|^p \|F-G\|_{\infty}^p\right)^{1/p}
    \\&\qquad= \|F - G\|_{\infty} \| K_{\mathbf{Z}} \|_{p}.
\end{align*}
It follows that
\begin{align}
\label{eq:cond-bound}
&\mathbb P\left(|\mathcal{Y}_F(\mathbf Z) - \mathcal{Y}_G(\mathbf Z)| > u \bigg| \mathbf{Z} =  \boldsymbol\xi  \right) \\ \nonumber &\qquad\leq \mathbb P\left( \|g\|_{q} > \frac{u \sqrt{N}}{\|K_{\mathbf{Z}}\|_{p} \|F-G\|_{\infty}} \bigg| \mathbf{Z}= \boldsymbol\xi \right).
\end{align}
It is straightfoward to see that $x \mapsto \|x\|_{q}$ is a Lipschitz function from $\mathbb R^N$ to $\mathbb R$. Then, by \citep[Theorem 5.6]{Boucheron}, $\|g\|_{q}$ satisfies the sub-Gaussian concentration inequality
%\begin{align}
\(
    \mathbb P\left( \|g\|_{q} > \epsilon\right) \leq 2 \exp\left(-\frac{\epsilon^2}{2 L^2} \right),
\)
%\end{align}
 where $L$ is defined above. Applying this to~\eqref{eq:cond-bound} completes the proof.
 \end{proof}

\subsection{Proof of Proposition~\ref{thm:gauss-complex}}
\begin{proof}{Proof}
It is straightforward to see that $\{\mathcal{Y}_F(\mathbf Z) : F \in \mathcal F\}$ is a separable random field. Further, by Lemma~\ref{lem:sub-gauss} $\{\mathcal{Y}_F(\mathbf Z) : F \in \mathcal F\}$ is sub-Gaussian. By Lemma~\ref{ass:diameter}, and the definition of the pseudometric $d$, we have $ D := \sup_{\xi_1,\xi_2 \in \mathcal{B}} d(\xi_1,\xi_2) < +\infty$. By Dudley's Theorem for separable random fields~\citep[Ch. 11]{Ledoux_1991} it follows that there exists a constant $0 < \mathtt{C}' < + \infty$ such that
%\begin{align*}
\(
    \mathbb E_{g} \left[ \sup_{F \in \mathcal F} |\mathcal{Y}_F(\mathbf Z) - \mathcal{Y}_{F_0}(\mathbf Z)| \bigg | \mathbf Z =  \boldsymbol\xi  \right] \leq \mathtt{C}' \int_0^{D/2} \sqrt{ \log N(\epsilon,\mathcal{B},d)} d\epsilon.
\)
%\end{align*}
By Lemma~\ref{lem:e-cover} in the Appendix (see Section~\ref{sec:coverproof}), it follows that $N(\epsilon,\mathcal{B},d) = N(\epsilon',\mathcal{F},\|\cdot\|_{\infty})$, where $\epsilon' = \epsilon \frac{\sqrt{N}}{\|K_{ \boldsymbol\xi }\|_{p}}$. Then, changing variables in the metric entropy integral above to $\epsilon'$, we have $D' = D \sqrt{N}/\|K_{\mathbf{Z}}\|_{p} = diam(\mathcal F)$ and
\begin{align*}
    &\int_0^{D/2} \sqrt{ \log N(\epsilon,\mathcal{B},d)} \,d\epsilon\\&\qquad= \frac{\|K_{ \boldsymbol\xi }\|_{p}}{\sqrt{N}} \int_0^{D'/2} \sqrt{\log N(\epsilon,\mathcal{F},\|\cdot\|_{\infty})}\, d \epsilon\\
     & \qquad \leq%\frac{\|K_{\mathbf{Z}}\|_{p}}{\sqrt{N}} \int_0^{\frac{1}{2} diam(\mathcal{F})} \epsilon^{-1/\alpha} d\epsilon\\
    \frac{\|K_{ \boldsymbol\xi }\|_{p}}{\sqrt{N}} \frac{\alpha}{\alpha-1} \left(\frac{1}{2}diam(\mathcal F)\right)^{\frac{\alpha-1}{\alpha}}.
\end{align*}
Note that by Lemma~\ref{ass:diameter} it follows that the right hand side above is finite. Setting $\mathtt{C} = \mathtt{C}' \frac{\alpha}{\alpha - 1}$ completes the proof.
 
\end{proof}

\subsection{Proof of Proposition~\ref{thm:equiconv}}
 Consider the following generalization of McDiarmid's inequality.
\begin{theorem}[Theorem 1~\citep{kontorovich2014}]~\label{thm:kont}
Let $(\mathcal{X},d,\pi)$ be a metric space that satisfies $\dsg(\mathcal{X}) <+\infty$, and $\varphi : \mathcal{X}^N \to \mathbb R$ is $1$-Lipschitz, then $E_\pi[\varphi(\mathbf Z)] < +\infty$, and
%\begin{align*}
\(
    \mathbb \pi\left( |\varphi(\mathbf Z) - \mathbb E_\pi[\varphi(\mathbf Z)]| > t \right) \leq 2 \exp\left( -\frac{t^2}{2N \dsg^2(\mathcal X)} \right),
\)
%\end{align*}
where $\mathbf Z = (Z_1,\cdots,Z_N)$ is an independent sample drawn from $\pi$.
\end{theorem}

Observe that this result significantly loosens the requirements in McDiarmid's inequality from boundedness to Lipschitz continuity. 

\sloppy
\begin{proof}{Proof}
We start by considering the functional $ \varphi : C^N \to \mathbb R$ defined as 
%\begin{align*}
\(
 \varphi( \boldsymbol\xi ) := \sup_{F \in \mathcal{F}} \left \{J(F) - \frac{1}{N} \sum_{i=1}^N \tilde J(\xi_i+F) \right\},   
\)
%\end{align*}
for any $ \boldsymbol\xi  \in C^N$. Let $ \boldsymbol\xi =(\xi_1,\cdots, \xi_N) \in C^N$, $ \boldsymbol\xi '=(\xi_1',\cdots, \xi_N') \in C^N$; the metric distance between these vectors of functions is given by $ \| \boldsymbol\xi - \boldsymbol\xi '\|=\sum_{i=1}^N \|\xi_i-\xi_i'\|_{\infty}$. Also, define the tuples $ \boldsymbol\xi ^1 = (\xi_1',\xi_2,\cdots,\xi_N)$, $ \boldsymbol\xi ^2 = (\xi_1',\xi_2', \xi_3,\cdots,\xi_N),~\ldots,~  \boldsymbol\xi ^N = (\xi_1',\xi_2',\ldots,\xi_N') \equiv  \boldsymbol\xi '$. Using the triangle inequality, it is straightforward to see that
%Let $\mathbf{Z}' := (Z_)$
\begin{align}
    |\varphi( \boldsymbol\xi )-\varphi( \boldsymbol\xi ')|
    % &=|\varphi( \boldsymbol\xi )-\varphi( \boldsymbol\xi ^1)+\varphi( \boldsymbol\xi ^1)-\varphi( \boldsymbol\xi ^2)\notag\\
    % &\quad+\cdots+\varphi( \boldsymbol\xi ^{N-1})-\varphi( \boldsymbol\xi ')|\\
    \label{eq:varphi-metric}
    &\leq |\varphi( \boldsymbol\xi )-\varphi( \boldsymbol\xi ^1)|+|\varphi( \boldsymbol\xi ^1)-\varphi( \boldsymbol\xi ^2)|\notag\\
    &\qquad+\cdots+|\varphi( \boldsymbol\xi ^{N-1})-\varphi( \boldsymbol\xi ')|,
\end{align}
where each pair of $ \boldsymbol\xi ^{k-1}$ and $ \boldsymbol\xi ^k$ differs only by the $k$th element. Let $ \boldsymbol\xi ^k(i)$ represent the $i^{\text{th}}$ element of the $k^{\text{th}}$ tuple and $F^* \in \mathcal{F}$ be the function that achieves the supremum in $\varphi( \boldsymbol\xi )$. For any such pair of vectors, we have 
\begin{align*}
%    &=\left|\sup_{F \in \mathcal{F}} \left \{J(F) - \frac{1}{N} \sum_{i=1}^N \tilde J( \boldsymbol\xi ^{k-1}(i) +F) \right\}-\sup_{F \in \mathcal{F}} \left \{J(F) - \frac{1}{N} \sum_{i=1}^N \tilde J( \boldsymbol\xi ^{k}(i) +F) \right\}\right|\\
&\left|\varphi( \boldsymbol\xi ^{k-1})-\varphi( \boldsymbol\xi ^{k})\right|\\
    =&\bigg| \sup_{F \in \mathcal{F}} \left \{J(F) - \frac{1}{N} \sum_{i=1}^N \tilde J( \boldsymbol\xi ^{k-1}(i) +F) \right\}\\ 
    &\qquad- \sup_{F \in \mathcal{F}} \left \{ \left(J(F) - \frac{1}{N} \sum_{i=1}^N \tilde J( \boldsymbol\xi ^{k-1}(i) +F) \right) \right. \\
    &\qquad\quad \left.+ \frac{1}{N} \left( \tilde J(\xi_k+F) - \tilde{J}(\xi_k'+F) \right)  \right\}\bigg|\\
    %&\leq \sup_{F \in \mathcal{F}}\left|\frac{1}{N} \sum_{i=1}^N \tilde J( \boldsymbol\xi ^{k-1}(i) +F)-\frac{1}{N} \sum_{i=1}^N \tilde J( \boldsymbol\xi ^{k}(i) +F) \right|\\
    %&\leq \sup_{F \in \mathcal{F}} \left|\frac{1}{N}\left( \tilde J(\xi_k+F)-\tilde J(\xi_k'+F) \right)\right|\\ 
     \leq&\,\frac{1}{N}\left|\left( \tilde J(\xi_k+F^*)-\tilde J(\xi_k'+F^*) \right)\right| \leq \frac{\kappa}{N}\|\xi_k-\xi_k'\|_{\infty},
\end{align*}
where the last inequality follows from Assumption~\ref{ass:LipschitzinZ}. Consequently, substituting this into~\eqref{eq:varphi-metric} we have
\begin{equation}
\begin{aligned}  
%\nonumber
|\varphi( \boldsymbol\xi )-\varphi( \boldsymbol\xi ')|
% &\leq |\varphi( \boldsymbol\xi )-\varphi( \boldsymbol\xi ^1)|+|\varphi( \boldsymbol\xi ^1)-\varphi( \boldsymbol\xi ^2)|\\
% &\qquad+\cdots+|\varphi( \boldsymbol\xi ^{N-1})-\varphi( \boldsymbol\xi ')|\\
\label{eq:varphi-lip}
&\leq\frac{\kappa}{N}\Big(\|\xi_1-\xi_1'\|_{\infty}+\|\xi_2-\xi_2'\|_{\infty}\\
&\qquad+\cdots+\|\xi_N-\xi_N'\|_{\infty}\Big)\\
&=\frac{\kappa}{N} \| \boldsymbol\xi - \boldsymbol\xi '\|.
\end{aligned}     
\end{equation}
In other words, the functional $\varphi$ is $\frac{\kappa}{N}$-Lipschitz continuous. 
%Combing the fact that $\dsg^2(C)<\infty$, we can apply ~\cite[Theorem 1]{Kontorovich} by Kontorovich to obtain:
\sloppy Now, by hypothesis we have $\dsg^2(C)<+\infty$, and therefore applying Theorem~\ref{thm:kont} we have
%\begin{align*}
\(
    \mathbb P\left(\varphi-\mathbb E(\varphi)>t\right) = \mathbb P\left(\frac{N}{\kappa}\left(\varphi-\mathbb E(\varphi) \right)>\frac{N}{\kappa}t\right)
    \leq \exp\left(-\frac{Nt^2}{2\kappa^2\dsg^2(C)} \right).
\)
%\end{align*}
Now, for any $\delta>0$, $\exp\left(-\frac{Nt^2}{2\kappa^2\dsg^2(C)}\right)\leq \delta$ implies that $t\geq \left(\frac{2\kappa^2\dsg^2(C)\log(1/\delta)}{N}\right)^{1/2}$. Hence, with probability at least $1-\delta$, we have
%\begin{align*}
\(
     \varphi<\mathbb E(\varphi)+ \left(\frac{2\kappa^2\dsg^2(C)\log(1/\delta)}{N}\right)^{1/2},
\)
~which yields the final expression in~\eqref{eq:sample-complex1}.
 \end{proof}

\section{Translating a Cover Under the Pseudometric to a Supremum‐Norm Cover}\label{sec:coverproof}
In the proof of Proposition~\ref{thm:gauss-complex}, we used the following result that an $\epsilon$-cover under the pseudometric can be ``translated'' into a corresponding $\epsilon$-cover under the supremum-norm.
\begin{lemma}~\label{lem:e-cover}
Fix $\epsilon > 0$. Let $ \boldsymbol\xi  = (\xi_1,\cdots,\xi_N)\in C^N$ and suppose that $B_1, \cdots, B_l \subset \mathbb R^N$ is an $\epsilon$-cover of 
$\mathcal{B} = \{\mathcal{G}_F( \boldsymbol\xi ) : F \in \mathcal{F}\}$ under the pseudometric~\eqref{eq:pmetric}. Then, there exist subsets $B_1',\cdots, B_l'$ that form an $\epsilon'$-cover of $\mathcal{F}$ under the supremum norm $\|\cdot\|_\infty$, with $\epsilon' = \frac{\epsilon \sqrt N}{ \|K_{ \boldsymbol\xi }\|_{p}}$.
\end{lemma}

\begin{proof}{Proof}
By definition, $B_i = \{y \in \mathcal{B} : d(y_i,y) \leq \epsilon\}$ for some $y_i \in \mathcal{B}$. Consider the set $\{F \in \mathcal F : \mathcal{G}_F( \boldsymbol\xi )\in B_i\} =: \tilde{B}_i$. For any $F \in \tilde{B}_i$, we have
%\begin{align}
 \(   d(y_i, \mathcal{G}_F( \boldsymbol\xi )) = \frac{1}{\sqrt{N}} \|K_{\mathbf{Z}}\|_{p} \|F_{y_i} - F\|_{\infty} \leq \epsilon.
 \)
%\end{align}
It follows that
\(
\|F_{y_i} - F\|_{\infty} \leq \frac{\epsilon \sqrt{N}}{\|K_{\mathbf{Z}}\|_{p}} = \epsilon'.
\)

Now, let $F' \in \mathcal{F}\setminus\cup_{i=1}^l \tilde{B}_i$. It follows that
\(
    \min_{1 \leq i \leq l} \|F_{y_i} - F'\| > \epsilon',
\)
 implying that $d(y_i,\mathcal{G}_{F'}( \boldsymbol\xi )) > \epsilon$. Therefore, $\mathcal{G}_{F'}( \boldsymbol\xi ) \not\in \cup_{i=1}^l B_i$. But, this is a contradiction since $B_1,\cdots,B_l$ is an $\epsilon$-cover of $\mathcal{B}$, implying that $\mathcal{F}\setminus\cup_{i=1}^l \tilde{B}_i = \emptyset$.
 
\end{proof}
\end{document}